\documentclass[12pt]{article}
\usepackage{amscd}
\usepackage{amsfonts,amssymb,amsthm}
\usepackage{mathtools}
\usepackage{url}
\usepackage{color}
\usepackage{enumerate}
\usepackage{bm}
\usepackage[top=30truemm,bottom=30truemm,left=20truemm,right=20truemm]{geometry}
\usepackage[all]{xy}
\usepackage[new]{old-arrows}
\usepackage{stmaryrd}
\usepackage{comment}
\usepackage{braket}
\usepackage{here}
\newtheorem{thm}{Theorem}[section]
\newtheorem{cor}[thm]{Corollary}
\newtheorem{lem}[thm]{Lemma}
\newtheorem{prop}[thm]{Proposition}

\theoremstyle{definition}
\newtheorem{defin}[thm]{Definition}
\newtheorem{rem}[thm]{Remark}
\newtheorem{exa}[thm]{Example}

\numberwithin{equation}{section}

\makeatletter
\newcommand{\subjclass}[2][1991]{%
  \let\@oldtitle\@title%
  \gdef\@title{\@oldtitle\footnotetext{#1 \emph{Mathematics subject classification.} #2}}%
}
\newcommand{\keywords}[1]{%
  \let\@@oldtitle\@title%
  \gdef\@title{\@@oldtitle\footnotetext{\emph{Key words and phrases.} #1.}}%
}
\makeatother

\newcommand{\GL}{\operatorname{GL}} 
\newcommand{\gl}{\operatorname{\mathfrak{gl}}}
\newcommand{\g}{\mathfrak{g}} 
\newcommand{\h}{\mathfrak{h}} 

\newcommand{\Lie}{\operatorname{Lie}}

\newcommand{\End}{\operatorname{End}}
\newcommand{\Aut}{\operatorname{Aut}}
\newcommand{\Hom}{\operatorname{Hom}}
\newcommand{\Ker}{\operatorname{Ker}}
\newcommand{\range}{\operatorname{Im}}

\newcommand{\tr}{\operatorname{tr}}
\newcommand{\sgn}{\operatorname{sgn}}
\newcommand{\unit}{\mathrm{Id}}
\newcommand{\pr}{\operatorname{pr}}
\newcommand{\spec}{\operatorname{Spec}} 
\newcommand{\res}{\operatorname*{res}}

\newcommand{\Z}{\mathbb{Z}}
\newcommand{\C}{\mathbb{C}}

\newcommand{\bM}{\operatorname{\mathbf{M}}}
\newcommand{\bA}{\mathbf{A}} 
\newcommand{\bd}{\mathbf{d}}
\newcommand{\bD}{\mathbf{D}} 
\newcommand{\bC}{\mathbf{C}} 
\newcommand{\bB}{\mathbf{B}}

\newcommand{\bV}{\mathbf{V}}

\newcommand{\bv}{\mathbf{v}}
\newcommand{\bw}{\mathbf{w}}

\newcommand{\sQ}{\mathsf{Q}}

\newcommand{\ov}{\overline}
\newcommand{\fp}[1]{\ensuremath{[\![#1]\!]} } % formal power series
\newcommand{\fl}[1]{\ensuremath{(\!(#1)\!)} } % formal Laurent series
\newcommand{\bbV}{\mathbb{V}}
\newcommand{\bbW}{\mathbb{W}}
\newcommand{\bmlam}{\bm{\lambda}}
\newcommand{\bmkap}{\bm{\kappa}}
\newcommand{\relmiddle}[1]{\mathrel{}\middle#1\mathrel{}}
\newcommand{\diag}{\operatorname{diag}}
\newcommand{\calO}{\mathcal{O}}
\newcommand{\ccalO}{\check{\mathcal{O}}}
\newcommand{\leg}{\mathrm{leg}}
\newcommand{\QS}{\mathcal{S}}
\newcommand{\calN}{\mathcal{N}}
\newcommand{\bfe}{\mathbf{e}}

\newcommand{\Stab}{\operatorname{Stab}}

\title{Symmetries of Quiver Schemes}
\author{Ryo Terada\thanks{Corresponding author}\ \thanks{Department of Mathematics, Faculty of Science Division I, Tokyo University of Science, 1-3
Kagurazaka, Shinjuku-ku, Tokyo 162-8601, Japan; \texttt{1123705@ed.tus.ac.jp}} 
\ and Daisuke Yamakawa\thanks{Department of Mathematics, Faculty of Science Division I, Tokyo University of Science, 1-3
Kagurazaka, Shinjuku-ku, Tokyo 162-8601, Japan; \texttt{yamakawa@rs.tus.ac.jp}}}
\subjclass[2020]{Primary 53D20; Secondary 53D30, 16G20, 20F55, 17B67}

\allowdisplaybreaks
\begin{document}
\mathtoolsset{showonlyrefs=true}
\maketitle

% \tableofcontents

\begin{abstract}
   We introduce reflection functors on quiver schemes in the sense of Hausel--Wong--Wyss, 
   generalizing those on quiver varieties. 
   Also we construct some isomorphisms between quiver schemes whose underlying quivers are different.
\end{abstract}

\subsection*{Keywords}
GLS preprojective algebra, Quiver scheme, Weyl group,
Reflection functor
\section{Introduction}

Let $\sQ$ be a finite quiver with no edge-loops 
and $\bd=(d_i)_{i \in I}$ be a collection 
of positive integers indexed by the vertex set $I$. 
We think of each $d_i$ as the ``multiplicity'' of $i$ 
and call the pair $(\sQ,\bd)$ a \emph{quiver with multiplicities}.

In \cite{YD1}, the second named author associated to 
$\bmlam \in \bigoplus_{i \in I} \C[\epsilon_i]/(\epsilon_i^{d_i})$, $\bv \in \Z_{\geq 0}^I$ 
a complex symplectic manifold
$\calN^{\mathrm{s}}_{\sQ,\bd}(\bmlam,\bv)$, 
called the quiver variety with multiplicities%
\footnote{In \cite{YD1} the parameter $\bmlam =(\lambda_i)$ 
is supposed to be an element of 
$\bigoplus_{i \in I} \epsilon_i^{-d_i}\C\fp{\epsilon_i}/\C\fp{\epsilon_i}$, 
but it may be regarded as an element of 
$\bigoplus_{i \in I} \C[\epsilon_i]/(\epsilon_i^{d_i})$ 
by multiplying each $\lambda_i$ by $\epsilon_i^{d_i}$.}. 
In the multiplicity-free case ($d_i=1$ for all $i$), 
it coincides with the quiver variety 
$\mathfrak{M}^\mathrm{reg}_\zeta(\bv,\bw)$
in the sense of Nakajima~\cite{NH1} with $\bw=0$, $\zeta = (0,\bmlam)$. 
% The purpose of that paper is 
% to construct analogues of quiver varieties such that they have affine Weyl group symmetry of type $C$.
% One of the main theorems in \cite{YD1} says that, in roughly speaking,  
% the quiver with multiplicities $(\sQ,\bd)$ determines a symmetrizable  
% (possibly non-symmetric) generalized Cartan matrix $\bC$, 
% and the quiver varieties with multiplicities 
% $\calN^{\mathrm{s}}_{\sQ,\bd}(\bmlam,\bv)$ for various $\bmlam,\bv$ 
% (we say they are \emph{of type} $\bC$) 
% admit symmetry of the associated Weyl group, 
% which coincides with the Weyl group symmetry of quiver varieties 
% generated by reflection functors~\cite{NH03} in the multiplicity-free case. 
% However, not all symmetrizable generalized Cartan matrices 
% % may be obtained in that way;
% occur as types of quiver varieties with multiplicities;
% for instance, affine Cartan matrices of type $C$ occur as particular examples, 
% but not of type $B$. 
% It was thus an interesting problem to find another way to 
% generalize Nakajima's quiver varieties so that all symmetrizable generalized 
% Cartan matrices occur as their types  
% because one may expect that the geometry of such generalization has 
% rich information about 
% the associated Kac--Moody algebras 
% and their quantum enveloping algebras like as quiver varieties 
% (see e.g.\ \cite{NH98,NH04}). 
One of the main theorems in \cite{YD1} is the following: 
Let $\bA =(a_{ij})_{i,j \in I}$ be the adjacency matrix of the quiver $\sQ$ 
and define $\bC' = 2\,\unit -\bA \diag (d_i)_{i \in I}$,
which is a symmetrizable (possibly non-symmetric) generalized Cartan matrix. 
Then, in roughly speaking, the quiver varieties with multiplicities 
$\calN^{\mathrm{s}}_{\sQ,\bd}(\bmlam,\bv)$ for various $\bmlam,\bv$ 
(we say they are \emph{of type} $\bC'$) 
admit symmetry of the associated Weyl group, 
which coincides with the Weyl group symmetry of quiver varieties 
generated by reflection functors~\cite{NH03} in the multiplicity-free case. 
Note that not all symmetrizable generalized Cartan matrices 
occur as types of quiver varieties with multiplicities;
for instance, affine Cartan matrices of type $C$ occur as particular examples, 
but not of type $B$. 
It was thus an interesting problem to find another way to 
generalize Nakajima's quiver varieties so that all symmetrizable generalized 
Cartan matrices occur as their types  
because one may expect that the geometry of such generalization has 
rich information about 
the associated Kac--Moody algebras 
and their quantum enveloping algebras like as quiver varieties 
(see e.g.\ \cite{NH98,NH04}). 

The paper~\cite{GLS} by Geiss--Leclerc--Schr\"oer could be 
expected to give a nice approach to the problem.
% On the other hand, Geiss--Leclerc--Schr\"oer~\cite{GLS} 
They associated an algebra $\Pi$ 
to \emph{every} symmetrizable generalized Cartan matrix $\bC$ with a symmetrizer.  
If $\bC$ is symmetric (with the trivial symmetrizer), 
then $\Pi$ coincides with the usual preprojective algebra of type $\bC$.
Recall that Nakajima's quiver variety $\mathfrak{M}^\mathrm{reg}_0(\bv,0)$ 
parametrizes isomorphism classes of irreducible representations 
of a preprojective algebra with dimension vector $\bv$. 
Thus their work leads to generalization of Nakajima's quiver varieties 
to the non-symmetric case.

Based on the work of Geiss--Leclerc--Schr\"oer, 
Hausel--Wong--Wyss~\cite{HWW} modified 
the definition of $\calN^{\mathrm{s}}_{\sQ,\bd}(\bmlam,\bv)$  
to introduce an affine scheme, called the \emph{quiver scheme}, 
which we denote by $\QS_{\sQ,\bd}(\bmlam,\bv)$ in this paper. 
In the multiplicity-free case, it also coincides with a quiver variety. 
See Remark~\ref{rem:QSvsQV} in the body of this paper 
on the difference between  
quiver schemes and quiver varieties with multiplicities. 
The purpose of this paper is to obtain analogues/generalization of  
the results obtained in \cite{YD1} for quiver schemes.

We briefly explain the main results in this paper. 
Following Geiss--Leclerc--Schr\"oer, 
we associate to each $(\sQ,\bd)$ a symmetrizable generalized Cartan matrix $\bC$   
in a manner different to \cite{YD1}; 
in this manner we may obtain all symmetrizable generalized Cartan matrices.  
We let the associated Weyl group 
act both on $\bigoplus_{i \in I} \C[\epsilon_i]/(\epsilon_i^{d_i})$ and on $\Z^I$. 
Thus for each $j \in I$, 
the $j$-th simple reflection gives rise to linear transformations 
$r_j \colon \bigoplus_{i \in I} \C[\epsilon_i]/(\epsilon_i^{d_i}) \to \bigoplus_{i \in I} \C[\epsilon_i]/(\epsilon^{d_i})$ 
and $s_j \colon \Z^I \to \Z^I$. 
The first main result generalizes reflection functors 
of Lusztig~\cite{Lus}, Maffei~\cite{Maf} and Nakajima~\cite{NH03} .

\begin{thm}[see Section~\ref{section:reflection}]
   Take $j \in I$, $\bmlam =(\lambda_i) \in \bigoplus_{i \in I} \C[\epsilon_i]/(\epsilon_i^{d_i})$, 
   $\bv \in \Z_{\geq 0}^I$ so that 
   $\lambda_j$ is a unit of $\C[\epsilon_j]/(\epsilon_j^{d_j})$. 
   Then there exists an isomorphism of schemes
   \[
      \mathcal{F}_j \colon \QS_{\sQ,\bd}(\bmlam,\bv) \xrightarrow{\sim} 
      \QS_{\sQ,\bd}(r_j(\bmlam),s_j(\bv)).
   \]
\end{thm}

Note that Geiss--Leclerc--Schr\"oer~\cite{GLS} also introduced reflection functors for $\Pi$ but 
we cannot use them to show the above since $\bmlam \neq 0$ by the assumption.
We conjecture that our $\mathcal{F}_j$'s satisfy the Coxeter relations of the Weyl group.

For instance, if $\bC$ is of type $A^{(1)}_1$, $A^{(1)}_2$, $A^{(1)}_3$, $D^{(1)}_4$ or $C^{(1)}_2$ 
and $\bv$ is a ``minimal positive imaginary root'', then 
$\QS_{\sQ,\bd}(\bmlam,\bv)$ for generic $\bmlam$, 
which in this case coincides with $\calN^{\mathrm{s}}_{\sQ,\bd}(\bmlam,\bv)$, 
may be regarded as the space of initial values of solutions to 
a Painlev\'e differential equation, 
and our reflection functors induce B\"acklund transformations of the equation; 
see \cite[Introduction]{YD1}. 
% It is interesting for us how about the case of type $B^{(1)}_2$.
It is interesting for us to ask the role of 
the reflection functors for quiver schemes of the other extended Dynkin types   
in the theory of integrable systems.

The second main result is a generalization of \cite[Theorem~5.8]{YD1}.

\begin{thm}[see Section~\ref{section:normalization}]
   Suppose that a sequence of pairwise distinct vertices, 
   which we denote by $0,1, \dots ,l$ ($l>0$), satisfies 
   the following conditions:   
   \begin{itemize}
      \item vertices $i,j$ in $\{ 0,1, \dots ,l \}$ are connected by exactly one arrow if $\lvert i-j \rvert =1$, and otherwise no arrow connects them; 
      \item no arrow connects any $i \in I \setminus \{ 0,1, \dots ,l \}$ and $j \in \{ 1,2, \dots ,l \}$;
      \item $d_0=1$ and $d_i=d$ ($i=1,2, \dots ,l$) for some integer $d>1$.
   \end{itemize}
   Also, suppose that a pair $(\bmlam,\bv) \in \bigoplus_{i \in I} \C[\epsilon_i]/(\epsilon_i^{d_i}) \times \Z_{\geq 0}^I$ satisfies  
   the following conditions:
   \begin{itemize}
      \item the sequence $v_0,v_1, \dots ,v_l$ is non-increasing;
      \item $\lambda_i(0) + \lambda_{i+1}(0) + \cdots + \lambda_j(0) \neq 0$ for all pairs $i \leq j$ in $\{ 1,2, \dots ,l \}$.
   \end{itemize}
   Then there exist another quiver with multiplicities $(\check{\sQ},\check{\bd})$ 
   with the same vertex set $I$ and a pair 
   $(\check{\bmlam},\check{\bv}) \in \bigoplus_{i \in I} \C[\epsilon_i]/(\epsilon_i^{\check{d}_i}) \times \Z_{\geq 0}^I$
   such that 
   $\QS_{\sQ,\bd}(\bmlam,\bv)$ and 
   $\QS_{\check{\sQ},\check{\mathbf{d}}}(\check{\bmlam},\check{\bv})$ 
   are isomorphic. 
\end{thm}

In fact, both $(\check{\sQ},\check{\bd})$ and $(\check{\bmlam},\check{\bv})$ 
are explicitly given and $(\check{\sQ},\check{\bd})$ 
does not depend on $(\bmlam,\bv)$. 
Using this theorem we can show that some quiver schemes are (affine) algebraic varieties.

This paper is organized as follows: In Section $2$, we 
recall the definition of preprojective algebra $\Pi$ 
in the sense of Geiss--Leclerc--Schr\"oer and 
quiver schemes $\QS_{\sQ,\bd}(\bmlam,\bv)$. 
Also, we recall some result of Hausel--Wong--Wyss on coadjoint orbits, 
which we will use to prove our second main theorem.  
Sections $3$ and $4$ are devoted to prove our first and second main theorems, 
respectively. 
% In Section $5$ (Appendix) we give the complete list of 
% the graphs with multiplicities 
% corresponding to symmetrizable generalized Cartan matrices 
% of finite/affine type with minimal symmetrizer;  
% we hope that it will be useful for the readers.

Throughout the paper, we write $\otimes$ for $\otimes_{\C}$.

\section{Quiver schemes}\label{sec:quiver-scheme}

In this section we recall the definitions of 
preprojective algebras in the sense of Geiss--Leclerc--Schr\"oer~\cite{GLS} 
and quiver schemes introduced by Hausel--Wong--Wyss~\cite{HWW}.

\subsection{Preliminaries}\label{subsec:pre}

In this subsection we introduce some symplectic vector spaces
related to truncated polynomial rings;
they are building blocks of quiver schemes.

For a positive integer $d$, put
\[
   R_d \coloneqq \C\fp{\epsilon}/\epsilon^d \C\fp{\epsilon},
   \quad R^d \coloneqq \epsilon^{-d}\C\fp{\epsilon}/\C\fp{\epsilon}
   \subset \C\fl{\epsilon}/\C\fp{\epsilon}.
\]
We also denote the variable $\epsilon$ by $\epsilon_d$
in order to distinguish it from other variables.
The bilinear form
\[
   \C\fp{\epsilon} \times \epsilon^{-d}\C\fp{\epsilon} \to \C;
   \quad (f,g) \mapsto
   \res_{\epsilon=0} \left( f(\epsilon)g(\epsilon)\mathrm{d}\epsilon \right)
\]
induces a non-degenerate pairing $R_d \times R^d \to \C$,
by which we may identify 
the vector space $R^d$ with the $\C$-dual space $R_d^*$ of $R_d$.
On the other hand, the multiplication by $\epsilon^d$ induces
a $\C$-linear isomorphism $R^d \simeq R_d$.
Thus we may also identify $R_d$ with $R_d^*$;
the corresponding pairing $R_d \times R_d \to \C$ is 
\[
   (f,g) \mapsto \langle f,g \rangle_d \coloneqq
   \res_{\epsilon=0} \left( f(\epsilon)g(\epsilon)\,\frac{\mathrm{d}\epsilon}{\epsilon^d} \right).
\]
More generally, for homomorphisms
$X \colon \bbW \to \bbV$, $Y \colon \bbV \to \bbW$
between free finitely generated $R_d$-modules $\bbV, \bbW$, we define
\begin{equation}\label{eq:pairing}
   \langle X, Y \rangle_d
   = \res_{\epsilon=0} \left( \tr_{R_d}(XY)\,\frac{\mathrm{d}\epsilon}{\epsilon^d} \right)
   = \langle \tr_{R_d}(XY),1 \rangle_d,
\end{equation}
where $\tr_{R_d} \colon \End_{R_d}(\bbV) \to R_d$ is the trace.
It gives an isomorphism $\Hom_{R_d}(\bbV,\bbW) \simeq \Hom_{R_d}(\bbW,\bbV)^*$.

The $\C$-algebra $R_d$ is $d$-dimensional
with a basis $\{\,1,\epsilon,\ldots ,\epsilon^{d-1} \,\}$.
More generally, if $d$ is a multiple of some positive integer $c$,
the homomorphism
\[
   R_c \to R_d; \quad \epsilon_c \mapsto \epsilon_d^{d/c}
\]
makes $R_d$ into a free $R_c$-algebra with a basis
$\{\, 1, \epsilon_d, \dots ,\epsilon_d^{d/c -1}\,\}$.
In this manner we equip each $R_d$-module $\bbV$ with a structure of $R_c$-module.
\begin{lem}\label{lem:pr}
   Suppose that $\bbV$ is a free $R_d$-module and let 
   $N \colon \bbV \to \bbV$ be the multiplication by $\epsilon_d$. Then the map
   \[
      \pr_{c,d} \colon \End_{R_c}(\bbV) \to \End_{R_d}(\bbV); \quad
      Z \mapsto \sum_{k=0}^{d/c-1} N^k Z N^{d/c-1-k}
   \]
   is the transpose of the inclusion $\End_{R_d}(\bbV) \hookrightarrow \End_{R_c}(\bbV)$:
   \[
      \langle \pr_{c,d}(Z), Z' \rangle_d = \langle Z, Z' \rangle_c
      \quad (Z \in \End_{R_c}(\bbV),\ Z' \in \End_{R_d}(\bbV)).
   \]
\end{lem}

\begin{proof}
   Since $\pr_{c,d}(Z)Z' = \pr_{c,d}(ZZ')$ for
   $Z \in \End_{R_c}(\bbV)$, $Z' \in \End_{R_d}(\bbV)$,
   it suffices to show
   \[
      \res_{\epsilon_d =0} \left( \tr_{R_d} \left( \pr_{c,d}(Z) \right) \frac{\mathrm{d}\epsilon_d}{\epsilon_d^d} \right)
      = \res_{\epsilon_c =0} \left( \tr_{R_c}(Z) \frac{\mathrm{d}\epsilon_c}{\epsilon_c^c} \right)
      \quad (Z \in \End_{R_c}(\bbV)).
   \]
   Take an ordered $R_d$-basis $(v_1,v_2, \dots ,v_n)$ of $\bbV$
   and let $(Z_{ij}) \in M_n(R_d)$ be the matrix representation of
   $\pr_{c,d}(Z)$.
   Also, let $(Z_{(i,k)(j,l)}) \in M_{nd/c}(R_c)$ be
   the matrix representation of $Z$
   with respect to the $R_c$-basis
   $v_{i,k} \coloneqq \epsilon_d^k v_i$,
   $i=1,2, \dots ,n$, $k=0,1, \dots ,d/c-1$ of $\bbV$.
   Then
   \[
      Z_{ij} = \sum_{k,l=0}^{d/c-1} \epsilon_d^{d/c-1-l+k} \left. Z_{(i,k)(j,l)}\right|_{\epsilon_c = \epsilon_d^{d/c}}.
   \]
   Notice that 
   \[
      \res_{\epsilon_d =0} \left(  f\frac{\epsilon_c}{\epsilon_d}\frac{\mathrm{d}\epsilon_d}{\epsilon_d^d} \right)
      = \res_{\epsilon_c =0} \left( f\frac{\mathrm{d}\epsilon_c}{\epsilon_c^c} \right)\quad
      \left(f=\left(\sum f_k\epsilon_c^k\right)\in R_{c}\right),
   \]
   thus the formula one easily deduces
   \[
      \res_{\epsilon_d =0} \left( Z_{ij}\, \frac{\mathrm{d}\epsilon_d}{\epsilon_d^d} \right)
      =
      \sum_{k=0}^{d/c-1} \res_{\epsilon_c =0} \left( Z_{(i,k)(j,k)}\, \frac{\mathrm{d}\epsilon_c}{\epsilon_c^c} \right).
   \]
   Hence
   \begin{align*}
      \res_{\epsilon_d =0} \left( \tr_{R_d} \left( \pr_{c,d}(Z) \right) \frac{\mathrm{d}\epsilon_d}{\epsilon_d^d} \right)
       & = \sum_{i=1}^n \res_{\epsilon_d =0} \left( Z_{ii}\, \frac{\mathrm{d}\epsilon_d}{\epsilon_d^d} \right)                            \\
       & = \sum_{i=1}^n \sum_{k=0}^{d/c-1} \res_{\epsilon_c =0} \left( Z_{(i,k)(i,k)}\, \frac{\mathrm{d}\epsilon_c}{\epsilon_c^c} \right)
      = \res_{\epsilon_c =0} \left( \tr_{R_c}(Z) \frac{\mathrm{d}\epsilon_c}{\epsilon_c^c} \right).
   \end{align*}
\end{proof}

For a finite dimensional $\C$-vector space $V$,
define
\[
   G_d(V) = \Aut_{R_d}(V \otimes R_d), \quad
   \g_d(V) = \End_{R_d}(V \otimes R_d).
\]
Since $G_d(V) \subset \GL_\C(V \otimes R_d)$ is the centralizer
of the multiplication by $\epsilon$,
it is a linear algebraic group with Lie algebra $\g_d(V)$.
We have an obvious isomorphism $\g_d(V) \simeq \gl_\C(V) \otimes R_d$, 
which enables us to identify each element of $\g_d(V)$ with a matrix polynomial 
\[
   \xi = \sum_{k=0}^{d-1} \xi_k \epsilon^k, \quad \xi_k \in \gl_\C(V).
\]
For instance, the identity $\unit_{V \otimes R_d}$ is identified with 
$\unit_V$. 
As a subset of $\g_d(V)$, 
the group $G_d(V)$ consists of all 
$g = \sum_{k=0}^{d-1} g_k \epsilon^k \in \g_d(V)$ such that $\det g_0 \neq 0$. 
Also, the pairing~\eqref{eq:pairing} for $\bbW=\bbV$ enables us to
identify $\g_d(V)$ with its $\C$-dual space.

Now let $V, W$ be two finite dimensional $\C$-vector spaces
and $d,c$ be positive integers with $c \mid d$.
Put $\bbV = V \otimes R_d$, $\bbW = W \otimes R_c$ and
consider the vector space
\[
   \bM \coloneqq \Hom_{R_c}(\bbW,\bbV) \oplus \Hom_{R_c}(\bbV,\bbW),
\]
together with the action of the linear algebraic group $G_d(V) \times G_c(W)$ defined by 
\[
   (g,h) \colon (X,Y) \mapsto (gXh^{-1},hYg^{-1}).
\]
Since $\bbV$ is also free as an $R_c$-module,
the two-form
\[
   \omega \coloneqq \langle \mathrm{d}X \wedge \mathrm{d}Y \rangle_c =
   \res_{\epsilon_c =0} \left(\tr_{R_c} (\mathrm{d}X \wedge \mathrm{d}Y)\frac{\mathrm{d}\epsilon_c}{\epsilon_c^c} \right)
\]
is a $G_d(V) \times G_c(W)$-invariant symplectic form on $\bM$.
Here, in the above $X$ denotes the first projection $\bM \to \Hom_{R_c}(\bbW,\bbV)$ 
regarded as a $\Hom_{R_c}(\bbW,\bbV)$-valued function on
$\bM$ and $Y$ denotes the second projection.  

The symplectic form $\omega$ has another description.
The extension of scalar gives an isomorphism
\[
   \Hom_{R_c}(\bbW,\bbV) \xrightarrow{\sim} \Hom_{R_d}(\bbW \otimes_{R_c} R_d,\bbV),
\]
which we denote by $X \mapsto X^{R_d}$.
Furthermore, the projection
\[
   \bbW \otimes_{R_c} R_d
   = \bigoplus_{k=0}^{d/c-1} \bbW \epsilon_d^k \to \bbW \epsilon_d^{d/c-1}
   \simeq \bbW
\]
induces an isomorphism
\begin{equation}\label{eq:iso}
   \Hom_{R_d}(\bbV,\bbW \otimes_{R_c} R_d) \xrightarrow{\sim} \Hom_{R_c}(\bbV,\bbW),
\end{equation}
whose inverse is explicitly described as
\[
   Y \mapsto Y^{R_d} \colon \bbV \ni v \mapsto
   \sum_{k=0}^{d/c-1} Y(\epsilon_d^{d/c-1-k}v) \otimes \epsilon_d^k.
\]
Observe that for $X \in \Hom_{R_c}(\bbW,\bbV)$ and $Y \in \Hom_{R_c}(\bbV,\bbW)$,
we have
\[
   X^{R_d} Y^{R_d} = \sum_{k=0}^{d/c-1} N^k XY N^{d/c-1-k}
   = \pr_{c,d}(XY).
\]
Thus the previous lemma shows
\[
   \langle X,Y \rangle_c = \langle X^{R_d},Y^{R_d} \rangle_d
   \quad ((X,Y) \in \bM),
\]
and hence
\[
   \omega = \langle \mathrm{d}X^{R_d} \wedge \mathrm{d}Y^{R_d} \rangle_d.
\]

\begin{prop}\label{prop:moment}
   The map
   \[
      \mu \colon \bM \to \g_d(V) \simeq \g_d(V)^*;
      \quad (X,Y) \mapsto X^{R_d} Y^{R_d}
   \]
   is a moment map generating the $G_d(V)$-action.
\end{prop}

\begin{proof}
   We have
   \[
      \omega = -\langle \mathrm{d}Y^{R_d} \wedge \mathrm{d}X^{R_d} \rangle_d
      = -\mathrm{d} \langle Y^{R_d}, \mathrm{d}X^{R_d} \rangle_d.
   \]
   Also, the generating vector fields $\xi^*$, $\xi \in \g_d(V)$
   are given by $\xi^*_{(X,Y)} = (\xi X,-Y\xi)$.
   Hence the moment map $\mu \colon \bM \to \g_d(V)$
   with $\mu(0,0)=0$ is
   \[
      \langle \mu(X,Y),\xi \rangle
      = \langle Y^{R_d}, \mathrm{d}X^{R_d}(\xi^*) \rangle_d
      = \langle Y^{R_d}, \xi X^{R_d} \rangle_d
      = \langle X^{R_d} Y^{R_d}, \xi \rangle_d
      \quad (\xi \in \g_d(V)).
   \]
\end{proof}

\begin{rem}\label{rem:loop-moment}
   Through the isomorphism 
   \[
     \bM \simeq \Hom_{R_d}(\bbW \otimes_{R_c} R_d,\bbV) \oplus \Hom_{R_d}(\bbV,\bbW \otimes_{R_c} R_d), 
   \]
   the action of $G_c(W)$ extends to an action of $G_d(W)=\Aut_{R_d}(\bbW \otimes_{R_c} R_d)$. 
   This action is Hamiltonian with moment map 
   \[
     \mu' \colon \bM \to \g_d(W); 
     \quad (X,Y) \mapsto -Y^{R_d} X^{R_d}.
   \]
   Since $\g_d(W) \simeq \g_c(W) \otimes_{R_c} R_d \simeq \gl_\C(W) \otimes R_d$, 
   we may also identify the dual space $\g_d(W)^*$ with 
   \[
     \gl_\C(W) \otimes R^d \simeq \g_c(W) \otimes_{R_c} R^d, 
   \] 
   where we regard $R^d$ as a free $R_d$-module of rank one 
   using the linear isomorphism $\epsilon^{-d} \colon R_d \xrightarrow{\sim} R^d$.
   Under this identification, the moment map $\nu$ is expressed as 
   \[
     \mu'(X,Y) = -\sum_{k=0}^{d/c-1} Y N^k X \otimes \epsilon_d^{d/c-d-k-1}.
   \]
   When $c=1$, such a moment map appears in \cite{YD1,YD2,YD3}.  
\end{rem}

\subsection{GLS preprojective algebras and quiver schemes}

In this paper a quiver $\sQ$ is always assumed to be finite,
and usually denoted as $\sQ =(I,\Omega,s,t)$,
where $I$ is the set of vertices,
$\Omega$ is the set of arrows,
and $s,t \colon \Omega \to I$ are the source/target maps.
For a quiver $\sQ =(I,\Omega,s,t)$, we denote by $\overline{\sQ}=(I,\ov{\Omega},s,t)$
the quiver obtained by reversing the orientation of each arrow of $\sQ$;
so for each $h \in \Omega$ we have the reversed arrow $\ov{h} \in \ov{\Omega}$,
satisfying $s(\ov{h})=t(h)$, $t(\ov{h})=s(h)$.
Putting together all the arrows of $\sQ$ and $\ov{\sQ}$
we get a quiver $\sQ+\overline{\sQ}=(I,H,s,t)$ with
arrow set $H \coloneqq \Omega \sqcup \ov{\Omega}$,
called the \emph{double} of $\sQ$,
together with an involution $H \to H$, $h \mapsto \ov{h}$.
We define a map $\sgn \colon H \to \{ \pm 1\}$ by $\sgn |_{\Omega} \equiv 1$, $\sgn |_{\ov{\Omega}} \equiv -1$.

\begin{defin}
   A \emph{quiver with multiplicities} is a quiver $\sQ$
   with each vertex $i$ equipped with a positive integer $d_i$,
   called the \emph{multiplicity} of $i$.
\end{defin}

Take a quiver $\sQ=(I,\Omega,s,t)$ with multiplicities $\bd=(d_i)_{i \in I}$.
For $i, j \in I$, put
\[
   d_{ij} \coloneqq \gcd(d_i,d_j),
   \quad f_{ij} \coloneqq \frac{d_j}{d_{ij}},
\]
and for $h \in H$, put
\[
   d_h \coloneqq d_{s(h)t(h)},
   \quad f_h \coloneqq f_{s(h)t(h)}.
\]
Let $\sQ' =(I,H',s,t)$ be the quiver obtained by
adding an edge-loop $\ell_i$ to the double $\sQ + \ov{\sQ}$ for each $i\in I$.
Let us recall the preprojective algebras
in the sense of Geiss--Leclerc--Schr\"{o}er~\cite{GLS}:

\begin{defin}
   The \emph{GLS preprojective algebra} $\Pi$ associated to the quiver with multiplicities $(\sQ,\bd)$
   is defined to be
   the quotient of the path algebra of $\sQ'$ modulo the following relations:
   \begin{enumerate}[({P}1)]
      \item $\ell_i^{d_i}=0$ for any $i\in I$;
      \item $\ell_{t(h)}^{f_h}h =h \ell_{s(h)}^{f_{\ov{h}}}$ for each arrow $h$ of $\sQ + \ov{\sQ}$;
      \item the \emph{mesh relations}
            \[
               \sum_{h \in H; t(h)=i}\sum_{k=0}^{f_h -1}\sgn(h)\ell^k_i h \overline{h} \ell_i^{f_h -1-k}=0
               \quad (i\in I).
            \]
   \end{enumerate}
\end{defin}

By definition, a (finite dimensional) representation of $\Pi$ 
is given by a datum consisting of
\begin{itemize}
   \item a finite dimensional $\C$-vector space $\bbV_i$ for each $i \in I$;
   \item a linear map $B_h \colon \bbV_i \to \bbV_j$ for each arrow $h \colon i \to j$ of $\sQ + \ov{\sQ}$;
   \item a linear transformation $N_i$ of $\bbV_i$ for each $i \in I$,
\end{itemize}
such that
\begin{enumerate}[({P'}1)]
   \item $N_i^{d_i}=0$ for any $i\in I$;
   \item $N_{t(h)}^{f_h}B_h =B_h N_{s(h)}^{f_{\ov{h}}}$ for each arrow $h$ of $\sQ + \ov{\sQ}$;
   \item the mesh relations
         \[
            \sum_{h \in H; t(h)=i}\sum_{k=0}^{f_h -1}\sgn(h) N^k_i B_h B_{\overline{h}} N_i^{f_h -1-k}=0
            \quad (i\in I).
         \]
\end{enumerate}
Observe that relations (P'1) make each $\bbV_i$ into a module over $R_{d_i}$,
and then relations (P'2) are equivalent to that
each $B_h \colon \bbV_{s(h)} \to \bbV_{t(h)}$ is an $R_{d_h}$-homomorphism.
In what follows we consider the case where
each $\bbV_i$ is a free $R_{d_i}$-module
(such a representation is said to be \emph{locally free}),
and take an $I$-graded $\C$-vector space
$\bV = \bigoplus_{i \in I} V_i$ so that
\[
   \bbV_i = V_i \otimes R_{d_i} \quad (i \in I).
\]
Then the tuple $\bB =(B_h)_{h \in H}$ lives in the vector space
\[
   \bM_{\sQ,\bd}(\bV) \coloneqq \bigoplus_{h \in H} \Hom_{R_{d_h}} (\bbV_{s(h)},\bbV_{t(h)}).
\]
Put
\[
   G_\bd(\bV) \coloneqq \prod_{i \in I} G_{d_i}(V_i), \quad
   \g_\bd(\bV) \coloneqq \Lie G_\bd(\bV) = \bigoplus_{i \in I} \g_{d_i}(V_i).
\]
For simplicity, we use the following notation for variables:
\[
   \epsilon_i \coloneqq \epsilon_{d_i} \in R_{d_i}, \quad
   \epsilon_h \coloneqq \epsilon_{d_h} \in R_{d_h}.
\]
The observations made in the previous subsection show that
$\bM_{\sQ,\bd}(\bV)$ has a symplectic form
\[
   \omega = \sum_{h \in \Omega} \braket{dB_h\wedge dB_{\overline{h}}}_{d_h}
   = \frac12 \sum_{h\in H}\sgn(h)\braket{dB_h\wedge dB_{\overline{h}}}_{d_h},
\]
and the obvious action of $G_\bd(\bV)$ on $\bM_{\sQ,\bd}(\bV)$ is
Hamiltonian with moment map
\begin{align*}
   \mu_{\bV}        & =(\mu_{\bV,i})_{i\in I} \colon
   \bM_{\sQ,\bd}(\bV) \to \g_d(\bV),             \\
   \mu_{\bV,i}(\bB) & \coloneqq \sum_{\substack{h\in H, \\t(h)=i}} \sgn(h) B_h^{R_{d_i}} B_{\overline{h}}^{R_{d_i}}
   = \sum_{\substack{h\in H,                            \\t(h)=i}} \sum_{k=0}^{f_h-1}\sgn(h)N_i^k B_h B_{\overline{h}} N _i^{f_h-k-1}. 
\end{align*}
Therefore the mesh relations (P'3) are exactly the same as the
moment map relation $\mu_{\bV}(\bB) = 0$.

Since two points on $\mu_{\bV}(\bB)$ are in the same $G_\bd(\bV)$-orbit
if and only if the corresponding representations of $\Pi$ are isomorphic,
we see that the isomorphism classes of locally free representations of $\Pi$
with fixed dimension vector are parametrized by
the orbit space $\mu_{\bV}^{-1}(0)/G_\bd(\bV)$.
Motivated by this observation, we define the quiver schemes 
in the sense of Hausel--Wong--Wyss~\cite{HWW} as follows:

\begin{defin} \label{quiverscheme}
   For $\bmlam =(\lambda_i) \in R_{\bd}\coloneqq \bigoplus_{i \in I} R_{d_i}$
   and a finite dimensional $I$-graded $\C$-vector space $\bV$,
   we define
   \[
      \QS_{\sQ,\bd}(\bmlam,\bv) = \spec \left(\C[\,\mu_{\bV}^{-1}(-\bmlam\,\unit_{\bV})\,]^{G_{\mathbf{d}}(\bV)}\right),
   \]
   where $\bv \coloneqq (\dim V_i)_{i \in I} \in \Z_{\geq 0}^I$
   is the dimension vector of $\bV$ and $\bmlam\,\unit_{\bV} \coloneqq (\lambda_i\,\unit_{V_i}) \in \g_\bd(\bV)^{G_\bd(\bV)}$.
   Also, for convenience we put $\QS_{\sQ,\bd}(\bmlam,\bv) = \emptyset$ for $\bv \in \Z^I \setminus \Z_{\geq 0}^I$.
   We call $\QS_{\sQ,\bd}(\bmlam,\bv)$ the \textit{quiver scheme} associated to $(\sQ,\bd)$ with dimension vector $\bv$ and complex parameter $\bmlam$.
\end{defin}

When $d_i=1$ for all $i \in I$,
the quiver scheme $\QS_{\sQ,\bd}(\bmlam,\bv)$
is Nakajima's quiver variety (with trivial real parameter).

\begin{rem}
   If $\bV$, $\bV'$ are two $I$-graded $\C$-vector spaces with 
   $\dim \bV = \dim \bV' = \bv$, 
   then we have a \emph{canonical} isomorphism 
   \[
      \spec \left(\C[\,\mu_{\bV}^{-1}(-\bmlam\,\unit_{\bV})\,]^{G_{\mathbf{d}}(\bV)}\right) 
      \simeq 
      \spec \left(\C[\,\mu_{\bV'}^{-1}(-\bmlam\,\unit_{\bV'})\,]^{G_{\mathbf{d}}(\bV')}\right).
   \]
   Thus we are identifying them, which is the reason why we use the notation  
   $\QS_{\sQ,\bd}(\bmlam,\bv)$ rather than $\QS_{\sQ,\bd}(\bmlam,\bV)$ 
   for the quiver scheme.

   Also, the isomorphism class of $\QS_{\sQ,\bd}(\bmlam,\bv)$ 
   does not depend on the orientation of the quiver $\sQ$.
\end{rem}

\begin{rem}\label{rem:value}
   Note that the subgroup 
   \[
     \C^\times \unit_{\bV} \coloneqq 
     \{\, (c\,\unit_{V_i})_{i \in I} \mid c \in \C^\times\,\} 
     \subset G_{\bd}(\bV) 
   \]
   trivially acts on $\bM_{\sQ,\bd}(\bV)$. 
   Hence the moment map $\mu_{\bV}$ takes values perpendicular to 
   the Lie algebra of $\C^\times \unit_{\bV}$, namely, 
   the image of $\mu_{\bV}$ is contained in 
   \[
     \g_{\bd}(\bV)_0 \coloneqq 
     \left\{\, (\xi_i)_{i \in I} \in \g_{\bd}(\bV) \relmiddle|  
     \sum_{i \in I} \res_{\epsilon_i =0} \left( \tr_{R_{d_i}}(\xi_i)\,\frac{\mathrm{d}\epsilon_i}{\epsilon_i^{d_i}} \right) = 0\,\right\}. 
   \]
   It follows that the quiver scheme 
   $\QS_{\sQ,\bd}(\bmlam,\bv)$ is empty unless 
   \begin{equation}\label{eq:level}
     \sum_{i \in I} v_i \res_{\epsilon_i =0} \left( \lambda_i\,\frac{\mathrm{d}\epsilon_i}{\epsilon_i^{d_i}} \right) = 0.
   \end{equation}
\end{rem}

\begin{rem}\label{rem:dim}
   Let $\bA=(a_{ij})_{i,j \in I}$ be the adjacency matrix 
   of the underlying graph of $\sQ$, i.e.,  
   \[
      a_{ij} = \#\{\, h \in H \mid s(h)=i,\, t(h)=j \,\},
   \]
   and put  
   \begin{equation}\label{eq:GCM}
      \bA'\coloneqq \left(\frac{a_{ij}}{d_{ij}}\right)_{i,j\in I} ,\quad 
      \bD\coloneqq \text{diag}(d_i)_{i\in I},\quad 
      \bC=(c_{ij})_{i,j\in I}\coloneqq 2\,\unit -\bA' \bD,
   \end{equation}
   where $\unit$ denotes the identity matrix. 
   Define a symmetric bilinear form 
   $(\phantom{v},\phantom{v})$ on $\Z^I$ by 
   \begin{equation}\label{eq:symm}
      (\bv,\bw) = {}^t \bv \bD \bC \bw \quad (\bv,\bw \in \Z^I). 
   \end{equation}
   If we formally apply the dimension formula for Hamiltonian reductions, 
   then the dimension of $\QS_{\sQ,\bd}(\bmlam,\bv)$ is equal to 
   \begin{align}
      \dim \bM_\sQ(\bV_{\bd}) - 2\dim \left( G_\bd(\bV)/\C^\times \unit \right) 
      &= \sum_{h \in H} \frac{v_{s(h)} d_{s(h)} v_{t(h)} d_{t(h)}}{d_h} -2\sum_{i \in I} v_i^2 d_i +2 \\
      &= \sum_{i,j \in I} \frac{a_{ij}}{d_{ij}} v_i d_i v_j d_j -2\sum_{i \in I} v_i^2 d_i +2 \\ 
      &= 2- {}^t \bv (2\,\bD - \bD \bA' \bD) \bv = 2 - (\bv,\bv). 
   \end{align}
   Hence the ``expected dimension'' of $\QS_{\sQ,\bd}(\bmlam,\bv)$ is equal to 
   $2 - (\bv,\bv)$, as in the case of quiver varieties.

   Note that if $\sQ$ has no edge-loops, then 
   $\bC$ is a symmetrizable generalized Cartan matrix with symmetrizer $\bD$, 
   and $(\phantom{v},\phantom{v})$ is 
   the standard symmetric bilinear form on the root lattice 
   (identified with $\Z^I$ using the basis consisting of the simple roots)
   associated to $\bD$. 
   Clearly $\bC$, $\bD$ do not depend on the orientation of the quiver $\sQ$, 
   so they only depend on the underlying ``graph with multiplicities''.
   All symmetrizable generalized Cartan matrices may be constructed in this way; 
   see \cite{GLS} for the inverse construction. 
   % and Section~\ref{sec:table}  
   % for the graphs with multiplicities 
   % corresponding to symmetrizable generalized Cartan matrices 
   % of finite/affine type with minimal symmetrizer.
\end{rem}

\begin{rem}\label{rem:QSvsQV}
   Let $\bV_{\bd}$ be the $I$-graded vector space 
   $\bigoplus_{i\in I} V_i \otimes R_{d_i}$
   and consider the symplectic vector space
   \[
      \bM_\sQ(\bV_{\bd}) \coloneqq \bigoplus_{h \in H} \Hom_\C(V_{s(h)} \otimes R_{d_{s(h)}},V_{t(h)} \otimes R_{d_{t(h)}})
   \]
   instead of $\bM_{\sQ,\bd}(\bV)$.
   The group $G_{\bd}(\bV)$ acts on $\bM_\sQ(\bV_{\bd})$ as a subgroup of
   \[
      \GL(\bV_{\bd}) \coloneqq \prod_{i \in I} \GL_\C(V_i \otimes R_i),
   \]
   and the quiver varieties with multiplicities introduced by
   the second named author~\cite{YD1}
   are defined as Hamiltonian reductions of $\bM_\sQ(\bV_{\bd})$
   by the action of $G_{\bd}(\bV)$.
   If $\gcd(d_i,d_j)=1$ for all $i,j \in I$ with $a_{ij}\geq 1$,
   then $\bM_{\sQ,\bd}(\bV) = \bM_{\sQ}(\bV_{\bd})$ and hence
   they are essentially the same as quiver schemes,
   although they are defined as
   complex manifolds (not schemes) using geometric invariant theory.
\end{rem}

\subsection{Some $G_d(V)$-coadjoint orbits}\label{subsec:HWW}

In this subsection we 
fix a finite dimensional $\C$-vector space $V$ 
together with a positive integer $d$, and 
review a result of Hausel--Wong--Wyss on some $G_d(V)$-coadjoint orbits.

Take any direct sum decomposition $V = \bigoplus_{i=0}^l W_i$ 
and elements $\theta_0,\theta_1, \dots ,\theta_l \in R_d$ ($l>0$) 
so that $\theta_i - \theta_j$ is a unit whenever $i \neq j$. 
Put 
\[
   \Theta \coloneqq \bigoplus_{i=0}^l \theta_i\,\unit_{W_i \otimes R_d} \in \g_d(V),
\]
and let $\calO_\Theta \subset \g_d(V)$ be the $G_d(V)$-coadjoint orbit of $\Theta$. 

Hausel--Wong--Wyss proved that $\calO_\Theta$ is an example of quiver schemes. 
Let $(\sQ,\bd)$ be the quiver consisting of $l$ vertices $\{ 1,2, \dots , l\}$ 
and $(l-1)$ arrows $h_i \colon i \to i+1$, $i=1,2, \dots ,l-1$ 
with multiplicities 
\[
   d_i = d \quad (i=1,2, \dots ,l)
\]
for some positive integer $d$. 
We call it the \emph{$d$-leg of length $l$}. 
Define a graded $\C$-vector space $\bV = \bigoplus_{i=1}^l V_i$ by 
\[
   V_i \coloneqq \bigoplus_{j \geq i} W_i \quad (i=1,2, \dots ,l),
\]
and consider the symplectic vector space 
\[
   \bM_{\sQ,\bd}(\bV) \oplus 
   \Hom_\C(V,V_1 \otimes R_d) \oplus \Hom_\C(V_1 \otimes R_d,V)
\]
acted (diagonally) on by the group $G_{\bd}(\bV)$. 
An element $\bB$ of this space consists of $R_d$-homomorphisms 
\[  
   B_{i+1,i} \colon V_i \otimes R_d \to V_{i+1} \otimes R_d, \quad 
   B_{i,i+1} \colon V_{i+1} \otimes R_d \to V_i \otimes R_d 
   \quad (i=1,2, \dots ,l-1),
\]
together with $\C$-linear maps 
\[ 
   a \colon V \to V_1 \otimes R_d, \quad 
   b \colon V_1 \otimes R_d \to V.
\]
For such an element we put 
\[
   B_{1,0} \coloneqq a^{R_d} \colon V \otimes R_d \to V_1 \otimes R_d,
   \quad B_{0,1} \coloneqq b^{R_d} \colon V_1 \otimes R_d \to V \otimes R_d. 
\]
Observe that the $G_{\bd}(\bV)$-action is Hamiltonian with moment map 
\[
   \tilde{\mu}_\bV = (\tilde{\mu}_{\bV,i}), \quad 
   \tilde{\mu}_{\bV,i}(\bB) = B_{i,i-1} B_{i-1,i} - B_{i,i+1} B_{i+1,i} 
   \quad (i=1,2, \dots ,l), 
\]
where $B_{l,l+1}, B_{l+1,l}$ are understood to be zero.

\begin{prop}[Hausel--Wong--Wyss~{\cite[Proposition~6.3.4]{HWW}}]\label{prp:HWW}
   Define $\bmlam =(\lambda_i) \in R_{\bd}$ by 
   \[
      \lambda_i = \theta_i - \theta_{i-1}.
   \] 
   % Then the $G_{\bd}(\bV)$-action on 
   % the level set $\tilde{\mu}_{\bV}^{-1}(-\bmlam\,\unit_{\bV})$ 
   % is free and has a geometric quotient.  
   % Moreover, the map 
   % \[
   %    \nu \colon \bM_{\sQ,\bd}(\bV) \oplus \Hom_\C(V,V_1 \otimes R_d) \oplus \Hom_\C(V_1 \otimes R_d,V) \to 
   %    \g_d(V); \quad \bB \mapsto -B_{0,1}B_{1,0} + \theta_0\,\unit_{V\otimes R_d}
   % \]
   % induces a symplectic isomorphism from 
   % $\tilde{\mu}_{\bV}^{-1}(-\bmlam\,\unit_{\bV})/G_{\bd}(\bV)$ 
   % to the coadjoint orbit $\calO_\Theta$. 
   Then the map 
   \[
      \nu \colon \bM_{\sQ,\bd}(\bV) \oplus \Hom_\C(V,V_1 \otimes R_d) \oplus \Hom_\C(V_1 \otimes R_d,V) \to 
      \g_d(V); \quad \bB \mapsto -B_{0,1}B_{1,0} + \theta_0\,\unit_{V\otimes R_d}
   \]
   restricts to a principal $G_{\bd}(\bV)$-bundle  
   $\tilde{\mu}_{\bV}^{-1}(-\bmlam\,\unit_{\bV}) \to \calO_\Theta$, 
   via which $\calO_\Theta$ is a Hamiltonian reduction of 
   $\bM_{\sQ,\bd}(\bV) \oplus \Hom_\C(V,V_1 \otimes R_d) \oplus \Hom_\C(V_1 \otimes R_d,V)$.
\end{prop}

See e.g.\ \cite{VP} for the definition of geometric quotient. 
Because our convention and the statement are slightly different to those of \cite{HWW}, 
we will give a proof below (our proof is different to the proof of \cite{HWW}).

\begin{lem}
   Let $\bB \in \tilde{\mu}_{\bV}^{-1}(-\bmlam\,\unit_{\bV})$. 
   Then $B_{i,i+1}$ is injective and $B_{i+1,i}$ is surjective 
   for all $i=0,1, \dots ,l-1$.
\end{lem}

\begin{proof}
   First consider the case of $d=1$. 
   In this case, \cite[Lemma~9.1]{CB} (with a different sign convention) 
   shows that the affine quotient 
   $\tilde{\mu}_{\bV}^{-1}(-\bmlam\,\unit_{\bV})/G_{\bd}(\bV)$ 
   is isomorphic to the closure of the orbit $\calO_\Theta \subset \gl_\C(V)$ 
   via the map $\nu$.
   Furthermore, it is known (see e.g.\ \cite[5.1.2, 5.1.4]{YD4}) 
   that the image of a point $\bB \in \tilde{\mu}_{\bV}^{-1}(-\bmlam\,\unit_{\bV})$ 
   lies in $\calO_\Theta$ if and only if 
   $B_{i,i+1}$ is injective and $B_{i+1,i}$ is surjective 
   for all $i=0,1, \dots ,l-1$. 
   Since $\calO_\Theta$ is closed ($\Theta$ is semisimple), 
   it follows that 
   $B_{i,i+1}$ is injective and $B_{i+1,i}$ is surjective 
   for all $\bB \in \tilde{\mu}_{\bV}^{-1}(-\bmlam\,\unit_{\bV})$ 
   and $i=0,1, \dots ,l-1$.

   Now consider the general case. 
   Take any $\bB \in \tilde{\mu}_{\bV}^{-1}(-\bmlam\,\unit_{\bV})$.
   Specializing $\epsilon$ to zero, 
   we then obtain $\C$-linear maps 
   \[  
      B_{i+1,i}(0) \colon V_i \to V_{i+1}, \quad 
      B_{i,i+1}(0) \colon V_{i+1} \to V_i 
      \quad (i=0,1, \dots ,l-1),
   \]
   where $V_0 \coloneqq V$, such that 
   \[
      B_{i,i-1}(0) B_{i-1,i}(0) - B_{i,i+1}(0) B_{i+1,i}(0) = -\lambda_i(0)\,\unit_{V_i}
      \quad (i=1,2, \dots ,l).
   \]
   Since $\theta_i(0) \neq \theta_j(0)$ whenver $i \neq j$, 
   the above fact in the case of $d=1$ shows that 
   $B_{i,i+1}(0)$ is injective and $B_{i+1,i}(0)$ is surjective  
   for all $i=0,1, \dots ,l-1$. 
   This implies that $B_{i,i+1}$ is injective and $B_{i+1,i}$ is surjective  
   for all $i=0,1, \dots ,l-1$. 
\end{proof}

\begin{proof}[Proof of Proposition~\ref{prp:HWW}]
   Using the above lemma one can easily check that $G_{\bd}(\bV)$ acts freely on 
   $\tilde{\mu}_{\bV}^{-1}(-\bmlam\,\unit_{\bV})$. 
   Hence $\tilde{\mu}_{\bV}^{-1}(-\bmlam\,\unit_{\bV})$ is non-singular 
   and equidimensional. 
   Also, note that the group $G_d(V)$ acts on the space 
   $\bM_{\sQ,\bd}(\bV) \oplus \Hom_\C(V,V_1 \otimes R_d) \oplus \Hom_\C(V_1 \otimes R_d,V)$
   in Hamiltonian fashion with moment map $\mu'$ (see Remark~\ref{rem:loop-moment}); 
   this action commutes with the $G_{\bd}(\bV)$-action and preserves $\tilde{\mu}_{\bV}^{-1}(-\bmlam\,\unit_{\bV})$.

   Now we take any $\bB \in \tilde{\mu}_{\bV}^{-1}(-\bmlam\,\unit_{\bV})$ and   
   show that $A \coloneqq \nu(\bB)$ lies in $\calO_\Theta$. 
   For $i=1,2, \dots ,l$, put 
   \[
      B_{0,i} = B_{0,1} B_{1,2} \cdots B_{i-1,i}, \quad 
      B_{i,0} = B_{i,i-1} \cdots B_{2,1} B_{1,0}.
   \]
   Then using the moment map relation iteratively one easily deduces 
   \begin{align*}
      \left(B_{0,1} B_{1,0} + (\lambda_1 + \cdots + \lambda_i)\unit_{V\otimes R_d} \right) B_{0,i} &= B_{0,i+1}B_{i+1,i} \quad (i=0,1, \dots ,l-1), \\
      \left(B_{0,1} B_{1,0} + (\lambda_1 + \cdots + \lambda_l)\unit_{V\otimes R_d} \right) B_{0,l} &= 0, 
   \end{align*}
   that is,
   \begin{equation}\label{eq:orbit1}
      \left(-A + \theta_i\,\unit_{V\otimes R_d} \right) B_{0,i} = B_{0,i+1}B_{i+1,i} 
      \quad (i=0,1, \dots ,l-1), \quad 
      \left(-A + \theta_l\,\unit_{V\otimes R_d} \right) B_{0,l} = 0.
   \end{equation}
   Thus for any $i=0,1, \dots ,l-1$, we have 
   \begin{align*}
      B_{0,i+1}B_{i+1,0} &= \left(-A + \theta_i \,\unit_{V\otimes R_d} \right) B_{0,i} B_{i,0} \\
      &= \left(-A + \theta_i \,\unit_{V\otimes R_d} \right)\left(-A + \theta_{i-1} \,\unit_{V\otimes R_d} \right) B_{0,i-1} B_{i-1,0} \\
      &= \cdots = \prod_{j=1}^i \left(-A + \theta_j\,\unit_{V\otimes R_d} \right) B_{0,1}B_{1,0} 
      = \prod_{j=0}^i \left(-A + \theta_j\,\unit_{V\otimes R_d} \right),
   \end{align*}
   and 
   \[
      \prod_{j=0}^l \left(-A + \theta_j\,\unit_{V\otimes R_d} \right) = 
      \left(-A + \theta_l\,\unit_{V\otimes R_d} \right) B_{0,l}B_{l,0} =0.
   \]
   By the above lemma, $B_{0,i}$ is injective and $B_{i,0}$ is surjective 
   for all $i=1,2, \dots ,l$. Hence 
   \begin{equation}\label{eq:orbit2}
      \range \prod_{j=0}^i \left(-A + \theta_j\,\unit_{V\otimes R_d} \right) \simeq V_{i+1} \otimes R_d 
      \quad (i=0,1, \dots ,l-1).
   \end{equation}
   Since $\theta_i - \theta_j \in R_d^\times$ ($i \neq j$) and $\prod_{j=0}^l \left( A - \theta_j\,\unit_{V\otimes R_d} \right) = 0$, 
   the idempotent decomposition  
   \[
      \unit_{V\otimes R_d} = \sum_{i=0}^l \pi_i, \quad 
      \pi_i \coloneqq \prod_{j \neq i} (\theta_i - \theta_j)^{-1} \prod_{j \neq i} (A - \theta_j\,\unit_{V\otimes R_d})
   \]
   gives a direct sum decomposition $V \otimes R_d = \bigoplus_{i=0}^l \Ker (A - \theta_i\,\unit_{V\otimes R_d})$, 
   and equalities~\eqref{eq:orbit2} show that $\Ker (A - \theta_i\,\unit_{V\otimes R_d}) \simeq W_i \otimes R_d$. 
   Hence $A \in \calO_\Theta$. Furthermore, if we put 
   \[
      \bbV_i \coloneqq \range \prod_{j=0}^{i-1} \left(-A + \theta_j\,\unit_{V\otimes R_d} \right) 
      = \range B_{0,i} 
      \quad (i=1,2, \dots ,l), 
   \]
   % \[
   %    \bbV_i(A) \coloneqq \range \prod_{j=0}^{i-1} \left(-A + \theta_j\,\unit_{V\otimes R_d} \right) 
   %    = \range B_{0,i} 
   %    \quad (i=1,2, \dots ,l), 
   % \]
   then equalities~\eqref{eq:orbit1} and the definition of $B_{0,i}$'s 
   yield the following commutative diagrams 
   for $i=0,1, \dots ,l-1$:
   \begin{equation}\label{eq:diagram}
      \vcenter{
      \xymatrix@C=4em{
         V_i \otimes R_d \ar[r]^-{B_{i+1,i}}\ar[d]_-{B_{0,i}}\ar@{}[rd]|{\circlearrowright} & V_{i+1} \otimes R_d \ar[d]^-{B_{0,i+1}}\\
         \bbV_i \ar[r]_-{-A+\theta_i\,\unit_{V\otimes R_d}} & \bbV_{i+1},
      }}
      \qquad 
      \vcenter{
      \xymatrix@C=4em{
         V_{i+1} \otimes R_d \ar[r]^-{B_{i,i+1}}\ar[d]_-{B_{0,i+1}}\ar@{}[rd]|{\circlearrowright} & V_i \otimes R_d \ar[d]^-{B_{0,i}}\\
         \bbV_{i+1} \ar[r]_-{\text{inclusion}} & \bbV_i.
      }}
   \end{equation}
   Here we use the conventions 
   $\bbV_0 = V \otimes R_d$, $V_0 = V$, $B_{0,0}=\unit_{V \otimes R_d}$ 
   % \[
   %    \bbV_i(A) \coloneqq \range \prod_{j=0}^{i-1} \left(-A + \theta_j\,\unit_{V\otimes R_d} \right) 
   %    = \range B_{0,i} 
   %    \quad (i=1,2, \dots ,l), 
   % \]
   % then equalities~\eqref{eq:orbit1} and the definition of $B_{0,i}$'s 
   % yield the following commutative diagrams 
   % for $i=0,1, \dots ,l-1$:
   % \[
   %    \xymatrix@C=5em{
   %       V_i \otimes R_d \ar[r]^-{B_{i+1,i}}\ar[d]_-{B_{0,i}}\ar@{}[rd]|{\circlearrowright} & V_{i+1} \otimes R_d \ar[d]^-{B_{0,i+1}}\\
   %       \bbV_i(A) \ar[r]_-{-A+\theta_i\,\unit_{V\otimes R_d}} & \bbV_{i+1}(A),
   %    }
   %    \qquad 
   %    \xymatrix@C=5em{
   %       V_{i+1} \otimes R_d \ar[r]^-{B_{i,i+1}}\ar[d]_-{B_{0,i+1}}\ar@{}[rd]|{\circlearrowright} & V_i \otimes R_d \ar[d]^-{B_{0,i}}\\
   %       \bbV_{i+1}(A) \ar[r]_-{\text{inclusion}} & \bbV_i(A).
   %    }
   % \]
   % Here we use the conventions 
   % $\bbV_0(A) = V \otimes R_d$, $V_0 = V$, $B_{0,0}=\unit_{V \otimes R_d}$ 
   in the case of $i=0$.
   Note that the vertical arrows are all isomorphisms, 
   and that $\bbV_i = V_i \otimes R_d$ for all $i$ if $A=\Theta$. 

   Define a point $\bB^0 \in \tilde{\mu}_{\bV}^{-1}(-\bmlam\,\unit_{\bV})$ so that  
   $B^0_{i,i+1} \colon V_{i+1} \otimes R_d \to V_i \otimes R_d$ are the inclusions 
   and 
   \[
     B^0_{i+1,i} = (-\Theta+\theta_i\,\unit_{V\otimes R_d})|_{V_i \otimes R_d} \colon V_i \otimes R_d \to V_{i+1} \otimes R_d
     \quad (i=0,1, \dots ,l-1).
   \]
   Since 
   $\nu \colon \tilde{\mu}_{\bV}^{-1}(-\bmlam\,\unit_{\bV}) \to \calO_\Theta$ 
   is $G_d(V)$-equivariant and $G_d(V)$ transitively acts on $\calO_\Theta$, 
   for any $\bB' \in \tilde{\mu}_{\bV}^{-1}(-\bmlam\,\unit_{\bV})$ 
   there exists $g \in G_d(V)$ such that $\nu(g \cdot \bB')=\Theta$, 
   and the vertical arrows of the above diagram for $\bB \coloneqq g \cdot \bB'$  
   define an element $(B_{0,i})_{i=1}^l$ of $G_{\bd}(\bV)$  
   sending $\bB$ to $\bB^0$. 
   It follows that $G_d(V) \times G_{\bd}(\bV)$ 
   transitively acts on 
   $\tilde{\mu}_{\bV}^{-1}(-\bmlam\,\unit_{\bV})$ and hence 
   \[
     \tilde{\mu}_{\bV}^{-1}(-\bmlam\,\unit_{\bV}) \simeq  
     \left( G_d(V) \times G_{\bd}(\bV) \right)/\Stab(\bB^0),
   \] 
   where $\Stab(\bB^0)$ is the stabilizer of $\bB^0$. 
   By a direct calculation, we see that $\Stab(\bB^0)$ consists of all the elements  
   $(g,(g_i))$ of $G_d(V) \times G_{\bd}(\bV)$ such that 
   \[
     g\Theta g^{-1}=\Theta,\quad g|_{V_i \otimes R_d} = g_i \quad (i=0,1, \dots ,l).
   \]
   Since any element of the centralizer $G_d(V)_\Theta \simeq \prod_{i=0}^l G_d(W_i)$ of $\Theta$ 
   preserves each $V_i \otimes R_d$,  
   the group $\Stab(\bB^0)$ is isomorphic to $G_d(V)_\Theta$ 
   via the map  
   \[
     G_d(V)_\Theta \to G_d(V) \times G_{\bd}(\bV), \quad 
     g \mapsto (g,(g|_{\bbV_i})).
   \]
   Thus $\tilde{\mu}_{\bV}^{-1}(-\bmlam\,\unit_{\bV})$  
   is a (geometric) quotient of $G_d(V) \times G_{\bd}(\bV)$ by the action of 
   $G_d(V)_\Theta$ defined by the above map and $\nu$ is identified with the first projection 
   \[
      \tilde{\mu}_{\bV}^{-1}(-\bmlam\,\unit_{\bV})
      =\left( G_d(V) \times G_{\bd}(\bV) \right)/G_d(V)_\Theta 
      \to G_d(V)/G_d(V)_\Theta = \calO_\Theta, 
   \]
   which is a principal $G_{\bd}(\bV)$-bundle. 
   % Hence the map 
   % $G_d(V) \times G_{\bd}(\bV) \to \tilde{\mu}_{\bV}^{-1}(-\bmlam\,\unit_{\bV})$, 
   % $(g,(g_i)) \mapsto (g,(g_i)) \cdot \bB^0$ 
   % is a principal $G_d(V)_\Theta$-bundle.
   % In particular, $\tilde{\mu}_{\bV}^{-1}(-\bmlam\,\unit_{\bV})$ 
   % is irreducible. 
   % Thus \cite[Theorem~4.2]{VP} shows that 
   % It follows that 
   % $\nu \colon \tilde{\mu}_{\bV}^{-1}(-\bmlam\,\unit_{\bV}) \to \calO_\Theta$ 
   % is a principal $G_{\bd}(\bV)$-bundle, 
   Hence $\calO_\Theta = \tilde{\mu}_{\bV}^{-1}(-\bmlam\,\unit_{\bV})/G_{\bd}(\bV)$ 
   carries a symplectic structure as a Hamiltonian reduction of 
   $\bM_{\sQ,\bd}(\bV) \oplus \Hom_\C(V,V_1 \otimes R_d) \oplus \Hom_\C(V_1 \otimes R_d,V)$. 
   However it coincides with the Kirillov--Kostant--Souriau symplectic structure 
   since $\nu$ is a moment map and hence preserves the Poisson structure.
   % the induced isomorphism 
   % $\tilde{\mu}_{\bV}^{-1}(-\bmlam\,\unit_{\bV})/G_{\bd}(\bV) \simeq \calO_\Theta$
   % preserves the Poisson structure (and hence the symplectic structure).
\end{proof}

In fact, the geometric quotient in 
Proposition~\ref{prp:HWW} is an example of quiver schemes. 
Let $(\widetilde{\sQ},\tilde{\bd})$ be the quiver with multiplicities 
obtained from $(\sQ,\bd)$ by 
adding a new vertex $0$ of multiplicity $1$ and $\dim V$ arrows from $0$ to $1$. 
Define a graded $\C$-vector space $\widetilde{\bV} = \bigoplus_{i=0}^m \widetilde{V}_i$ 
by $\widetilde{V}_0 \coloneqq \C$, $\widetilde{V}_i \coloneqq V_i$ ($i>0$).
Then fixing a linear isomorphism $V \simeq \C^{\dim V}$, we have  
\begin{align*}
  \bM_{\widetilde{\sQ},\tilde{\bd}}(\widetilde{\bV}) 
  &= \bM_{\sQ,\bd}(\bV) \oplus \bigoplus_{i=1}^{\dim V} \left( \Hom_\C(\C,V_1 \otimes R_d) \oplus \Hom_\C(V_1 \otimes R_d,\C) \right) \\
  &\simeq \bM_{\sQ,\bd}(\bV) \oplus \Hom_\C(V,V_1 \otimes R_d) \oplus \Hom_\C(V_1 \otimes R_d,V).
\end{align*}
Also, the moment map $\mu_{\widetilde{\bV}}$ is described as 
\[
  \mu_{\widetilde{\bV},0}(\bB) = - \tr B_{0,1}B_{1,0}, \quad 
  \mu_{\widetilde{\bV},i}(\bB) = B_{i,i-1}B_{i-1,i} - B_{i,i+1}B_{i+1,i} = \tilde{\mu}_{\bV,i}(\bB) 
  \quad (i>0).
\]
Let $\bmlam$ be as in Proposition~\ref{prp:HWW} and 
define $\tilde{\bmlam}=(\tilde{\lambda}_i) \in R_{\tilde{\bd}}$ by 
\[
   \tilde{\lambda}_i = \lambda_i \quad (i>0), \quad 
   \tilde{\lambda}_0 = - \frac{1}{\dim V} \sum_{i>0} v_i \res_{\epsilon_i =0} \left( \lambda_i\,\frac{\mathrm{d}\epsilon_i}{\epsilon_i^{d_i}} \right),
\]
so that \eqref{eq:level} holds. 
Then any $\bB \in \bM_{\widetilde{\sQ},\tilde{\bd}}(\widetilde{\bV})$ 
satisfying $\tilde{\mu}_{\bV}(\bB)=-\bmlam\,\unit_{\bV}$ 
also satisfies $\mu_{\widetilde{\bV},0}(\bB)=-\lambda_0$ as $\mu_{\widetilde{\bV}}(\bB)$ 
lives in $\g_{\tilde{\bd}}(\widetilde{\bV})_0$ (see Remark~\ref{rem:value}).  
Thus 
\[
  \mu_{\widetilde{\bV}}^{-1}(-\tilde{\bmlam}\,\unit_{\widetilde{\bV}}) 
  = \tilde{\mu}_{\bV}^{-1}(-\bmlam\,\unit_{\bV}). 
\] 
Furthermore, since 
$\C^\times \unit_{\widetilde{\bV}} \subset G_{\tilde{\bd}}(\widetilde{\bV})=\C^\times \times G_{\bd}(\bV)$ 
acts trivially, we have 
\[
  \C[\,\mu_{\widetilde{\bV}}^{-1}(-\tilde{\bmlam}\,\unit_{\tilde{\bV}})\,]^{G_{\tilde{\bd}}(\widetilde{\bV})} 
  = \C[\,\tilde{\mu}_{\bV}^{-1}(-\bmlam\,\unit_{\bV})\,]^{G_{\bd}(\bV)}.
\]

\begin{cor}\label{crl:orbit}
   The orbit $\calO_\Theta$ is isomorphic to the quiver scheme 
   $\QS_{\widetilde{\sQ},\tilde{\bd}}(\tilde{\bmlam},\tilde{\bv})$.
\end{cor}

\begin{proof}
   Proposition~\ref{prp:HWW} implies that 
   $\QS_{\widetilde{\sQ},\tilde{\bd}}(\tilde{\bmlam},\tilde{\bv})$ 
   is isomorphic to the affinization $\spec \C[\calO_\Theta]$ of $\calO_\Theta$. 
   On the other hand, $\calO_\Theta$ is known to be affine 
   (see \cite[Lemma~2.2.4]{HWW} or Corollary~\ref{crl:affine}).
   Thus $\spec \C[\calO_\Theta]=\calO_\Theta$.
\end{proof}

\section{Reflection functors for quiver schemes}\label{section:reflection}

In this section, 
we modify the arguments made in \cite[Section~4]{YD1} 
to generalize the reflection functors of Lusztig~\cite{Lus}, Maffei~\cite{Maf} and Nakajima~\cite{NH03} 
for quiver schemes. 
Fix a quiver with multiplicities $(\sQ,\bd)$ with $\sQ$ having no edge-loops.

\subsection{Reflection functors}

Let $\bC$, $\bD$ be the symmetrizable generalized Cartan matrix and the symmetrizer 
defined in Remark~\ref{rem:dim}.  
Fix a realization $(\h, \{\alpha_i\}_{i\in I}, \{\alpha^{\lor}_i\}_{i\in I})$ 
of $\bC$ in the sense of \cite{Kac}; 
so $\h$ is the Cartan subalgebra, $\{ \alpha_i \}_{i \in I} \subset \h^*$ 
is the set of simple roots, and 
$\{\alpha^{\lor}_i\}_{i\in I} \subset \h$ is the set of simple coroots.
Let $Q \coloneqq \sum_{i\in I} \Z \alpha_i$ be the root lattice 
and identify it with $\Z^I$ using the basis $\{ \alpha_i \}_{i \in I}$. 
Then the dimension vectors of finite dimensional $I$-graded $\C$-vector spaces 
live in the subset $Q_{+}\coloneqq \sum_{i\in I} \Z_{\geq0}\alpha_i$.

Recall that the Weyl group $W(\bC)$ of $\bC$ 
is the subgroup of $\GL_\C(\h^{*})$ generated by the simple reflections
\[
   s_i \colon \h^{*}\to \h^{*};\quad
   \bmlam \mapsto \bmlam -\braket{\bmlam,\alpha_i^{\lor}} \alpha_i
   \quad (i \in I).
\]
The group $W(\bC)$ is a Coxeter group with defining relations
\begin{equation}\label{eq:coxeter}
   s_i^2=\unit_{\h^*},\ (s_is_j)^{m_{ij}}=\unit_{\h^*}\quad (i,j\in I,\ i \neq j),
\end{equation}
where $m_{ij}$ are determined from $c_{ij}c_{ji}$ by the following table.
\begin{table}[H]
   \begin{center}
      \begin{tabular}{|c|c|c|c|c|c|}  \hline
         $c_{ij}c_{ji}$ & 0 & 1 & 2 & 3 & $\geq 4$ \\\hline
         $m_{ij}$       & 2 & 3 & 4 & 6 & $\infty$ \\\hline
      \end{tabular}
   \end{center}
\end{table}

We define a $W(\bC)$-action on $R_{\bd} \times Q$.
The action on the second component $Q$ is 
just the restriction of the action on $\h^*$; explicitly, 
\[
   s_i(\bv) = \bv-\sum_{j\in I}c_{ij}v_j\alpha_i, 
   \quad \bv = \sum_{j \in I} v_j \alpha_j \in Q
   \quad (i \in I).
\]
This action is effective; so we may regard $W(\bC)$ as a subgroup of $\GL_\Z(Q)$.
On the other hand, the action on the first component $R_{\bd}$ is defined by
\[
   r_i(\bmlam)=(r_i(\bmlam)_j)_{j\in I},\quad
   r_i(\bmlam)_j\coloneqq \begin{cases}
      -\lambda_i                                                                                                          & (j=i),     \\
      \lambda_j -\sum_{l=0}^{d_{ij}-1}\lambda_{i,(d_i-\frac{d_i}{d_{ij}}l-1)}c_{ij}\epsilon_j^{d_j-\frac{d_j}{d_{ij}}l-1} & (j\neq i),
   \end{cases}
\]
where $\bmlam =(\lambda_i)_{i \in I} \in R_\bd$, $\lambda_i = \sum_{k=0}^{d_i-1} \lambda_{i,k} \epsilon_i^k$.

\begin{prop}\label{prop:reflection}
   The above $r_i, i\in I$ satisfy relations $\eqref{eq:coxeter}$.
\end{prop}

\begin{proof}
   For each $i \in I$, we define a map $\tilde{s}_i \colon R_\bd \to R_\bd$ by 
   \[
      \tilde{s}_i(\bmkap) = \bmkap - \sum_{j \in I} c_{ij} \sum_{m=0}^{d_{ij}-1} \kappa_{j,f_{ij}m} \epsilon_i^{f_{ji}m} \bfe_i
      \quad \left(\bmkap =\left(\sum \kappa_{i,k} \epsilon_i^k \right) \in R_\bd\right),
   \]
   where $\bfe_i \coloneqq (\delta_{ij})_{j \in I}$. 
   Then we have
   \begin{flalign*}
      \sum_{j\in I}\langle \lambda_j, \tilde{s}_i(\bmkap)_j\rangle_{d_j}&=
      \res_{\epsilon_i =0}\left(\lambda_i\left(\kappa_i-\sum_{j\in I}c_{ij}\sum_{m=0}^{d_{ij}-1}\kappa_{j,f_{ij}m}\epsilon_i^{f_{ji}m}\right)\frac{\mathrm{d}\epsilon_i}{\epsilon_i^{d_i}}\right)+\sum_{j\neq i}\res_{\epsilon_j =0}\left(\left(\lambda_j\kappa_j\right)\frac{\mathrm{d}\epsilon_j}{\epsilon_j^{d_j}}\right)\\
      &=-\res_{\epsilon_i =0}\left(\left(\lambda_i\kappa_i\right)\frac{\mathrm{d}\epsilon_i}{\epsilon_i^{d_i}}\right)-\sum_{j\neq i}c_{ij}\sum_{m=0}^{d_{ij}-1}\lambda_{i,(d_i-f_{ji}m-1)}\kappa_{j,f_{ij}m}+\sum_{j\neq i}\res_{\epsilon_j =0}\left(\left(\lambda_j\kappa_j\right)\frac{\mathrm{d}\epsilon_j}{\epsilon_j^{d_j}}\right)\\
      &=-\res_{\epsilon_i =0}\left(\left(\lambda_i\kappa_i\right)\frac{\mathrm{d}\epsilon_i}{\epsilon_i^{d_i}}\right)+\res_{\epsilon_j =0}\left(\left(\lambda_j-\sum_{m=0}^{d_{ij}-1}\lambda_{i,(d_i-f_{ji}m-1)}\right)\kappa_j\frac{\mathrm{d}\epsilon_j}{\epsilon_j^{d_j}}\right)\\
      &=\sum_{j\in I}\langle  r_i(\bmlam)_j,\kappa_j\rangle_{d_j}.
   \end{flalign*}
   Hence the transpose of $r_i$ is equal to $\tilde{s}_i$.
   Put $\tilde{I} = \{\, (i,k) \mid i \in I,\ k=0,1, \dots ,d_i-1\,\}$ and 
   define a matrix $\widetilde{\bC} = \left( \tilde{c}_{(i,k)(j,l)} \right) \in \C^{\tilde{I} \times \tilde{I}}$
   so that 
   \[
      \tilde{s}_i(\bmkap)_{i,k} = \kappa_{i,k} - \sum_{(j,l) \in \tilde{I}} \tilde{c}_{(i,k)(j,l)} \kappa_{j,l};
   \]
   explicitly, 
   \[
      \tilde{c}_{(i,k)(j,l)} = 
      \begin{cases}
         c_{ij} & (\text{$k=f_{ji} m$, $l=f_{ij} m$ for some $m \in \Z$}), \\
         0 & (\text{otherwise}).
      \end{cases}
   \]
   Then the matrix $\widetilde{\bC}$ is a symmetrizable generalized Cartan matrix 
   with symmetrizer $\widetilde{\bD} = \diag (\tilde{d}_{i,k})$, where 
   $\tilde{d}_{i,k} \coloneqq d_i$. 
   For $(i,k) \in \tilde{I}$, let $s_{i,k} \colon \Z^{\tilde{I}} \to \Z^{\tilde{I}}$ 
   be the $(i,k)$-th simple reflection acting on the root lattice $\Z^{\tilde{I}}$ for $\widetilde{\bC}$.
   Then for any $i \in I$, the reflections $s_{i,0}, s_{i,1}, \dots ,s_{i,d_i-1}$ commute pairwise, 
   and $\tilde{s}_i$ coincides with the linear map
   \[
      (s_{i,0} s_{i,1} \cdots s_{i,d_i-1}) \otimes_\Z \unit_\C  \colon \C^{\tilde{I}} \to \C^{\tilde{I}}
   \] 
   under the obvious identification $\C^{\tilde{I}}=R_\bd$. 
   Now relations $\eqref{eq:coxeter}$ follow from the defining relations for the Weyl group $W(\widetilde{\bC})$.
\end{proof}

\begin{rem}\label{rem:trans}
   Define a linear map $\rho \colon R_\bd \to \C^I$ by 
   \[
      \rho \colon (\lambda_i)_{i \in I} \mapsto \left( \res_{\epsilon_i =0} \left( \lambda_i\,\frac{\mathrm{d}\epsilon_i}{\epsilon_i^{d_i}} \right)\right)_{i \in I}.
   \]
   Then one can easily check that $\rho (r_i(\bmlam)) = {}^t s_i (\rho (\bmlam))$ for all $i \in I$.
   In particular, if $d_i=1$ for all $i \in I$, then 
   the $W(\bC)$-action on $R_\bd = \C^I$ is dual to that on $Q \otimes_\Z \C$.
\end{rem}

\begin{exa}
   (i) Suppose that $(\sQ,\bd)$ has the graph with multiplicities given below
   \[
      \hfill
      \begin{xy}
         (0,0)*++!D{d}*++!U{j} *\cir<4pt>{};
         (10,0)*++!D{1}*++!U{i} *\cir<4pt>{};
         (20,0) *++!D{1}*++!U{k} *\cir<4pt>{};
         \ar@{-} (2,0);(8,0)
         \ar@{-} (12,0);(18,0)
      \end{xy}
      \hfill
   \]
   Here we assume $d >1$.
   Then the corresponding generalized Cartan matrix is
   \[
      2\unit -
      \begin{pmatrix}
         0&1&1\\
         1&0&0\\
         1&0&0\\
      \end{pmatrix}
      \begin{pmatrix}
         1&0&0\\
         0&d&0\\
         0&0&1\\
      \end{pmatrix}
      =
   \begin{pmatrix}
   2&-d&-1\\
   -1&2&0\\
   -1&0&2
   \end{pmatrix}.
   \]
   
   We have
   \[ r_i(\bmlam)_i=-\lambda_i,\quad
   r_i(\bmlam)_j=\lambda_i-c_{ij}\lambda_i\epsilon_j^{d-1}=\lambda_j+d\lambda_i\epsilon_j^{d-1},\quad
  r_i(\bmlam)_k=\lambda_k-c_{ik}\lambda_i =\lambda_i+\lambda_k.  \]
     It coincides with the ones in \cite[Section 4]{YD1}.
   In general, if $\gcd(d_i,d_j)=1$ for all $j\in I$ joining the vertex $i$, then the action coincides with the action in \cite[Section 4]{YD1}.

   (ii) Suppose that $(\sQ,\bd)$ has the graph with multiplicities given below
   \[
      \hfill
      \begin{xy}
        (0,0)*++!D{d}*++!U{j} *\cir<4pt>{};
         (10,0)*++!D{d}*++!U{i} *\cir<4pt>{};
        (20,0) *++!D{1}*++!U{k} *\cir<4pt>{};
        \ar@{-} (2,0);(8,0)   
        \ar@{-} (12,0);(18,0) 
      \end{xy}
      \hfill
      \]
      Here we assume $d >1$.
      Then the Cartan matrix of it is
      \[  2\unit -
      \begin{pmatrix}
         0&\frac{1}{d}&1\\
         \frac{1}{d}&0&0\\
         1&0&0\\
      \end{pmatrix}
      \begin{pmatrix}
         d&0&0\\
         0&d&0\\
         0&0&1\\
      \end{pmatrix}
      =
        \begin{pmatrix}
         2&-1&-1\\
         -1&2&0\\
         -d&0&2
        \end{pmatrix} .
      \]
      We have 
      \[
      r_i(\bmlam)_i=-\lambda_i,\quad
      r_i(\bmlam)_j=\lambda_j -c_{ij}\lambda_i=\lambda_j+\lambda_i,\quad
       r_i(\bmlam)_k =\lambda_k -c_{ik}\lambda_{i,d-1}=\lambda_k+\lambda_{i,d-1}.
      \]
\end{exa}

\begin{thm}\label{thm:reflection}
   Suppose that $\sQ$ has no edge-loops.  
   Take $\bmlam \in R_{\bd}$, $\bv \in Q_+$, $i \in I$ and 
   suppose that $\lambda_i \in R_{d_i}$ is a unit.
   Then there exists an isomorphism of $\C$-schemes
   \[
      \mathcal{F}_i \colon \QS_{\sQ,\bd}(\bmlam,\bv) \xrightarrow{\sim} \QS_{\sQ,\bd}(r_i(\bmlam),s_i(\bv)).
   \]
\end{thm}
The map $\mathcal{F}_i$ generalizes the $i$-th
reflection functor of quiver varieties~\cite{NH03}.

We will prove this theorem in the next subsection. 
\begin{rem}
In fact, the above theorem holds even if $\sQ$ has edge-loops on vertices except $i$.
However, for simplicity, we suppose that all the vertices have no edge-loops 
because we are mainly interested in the case where $(\sQ, \bd)$ defines a symmetrizable generalized Cartan matrix.    
\end{rem}
\subsection{Proof of Theorem~\ref{thm:reflection}}\label{subsec:ref-proof}
For $h\in H$, define
\(\displaystyle V_h\coloneqq \sum_{l=0}^{f_{\overline{h}}-1}V_{s(h)}\epsilon_{s(h)}^l\)
so that $V_h \otimes R_{d_h} = V_{s(h)} \otimes R_{d_{s(h)}}$.
Then the extension of scalar
gives an isomorphism
\[
   \alpha_h \colon \Hom_{\C}(V_h,V_{t(h)}\otimes R_{d_{t(h)}})
   \xrightarrow{\sim} 
   \Hom_{R_{d_h}}(V_{s(h)}\otimes R_{d_{s(h)}},V_{t(h)}\otimes R_{d_{t(h)}}).
   \]
As similarly to \eqref{eq:iso}, we also obtain an isomorphism 
\[
   \beta_h \colon \Hom_{\C}(V_{s(h)}\otimes R_{d_{s(h)}},V_{\overline{h}})
   \xrightarrow{\sim}
   \Hom_{R_{d_h}}(V_{s(h)}\otimes R_{d_{s(h)}},V_{t(h)}\otimes R_{d_{t(h)}}).
\]
Fix a vertex $i \in I$ and set
$\widetilde{V}_i\coloneqq \bigoplus_{t(h)= i}V_h$, so 
\[
   \dim \widetilde{V}_i = \sum_{t(h)=i} \dim V_h 
   = \sum_{t(h)=i} f_{\ov{h}} v_{s(h)} 
   = \sum_{j \in I} a_{ij} \frac{d_j}{d_{ij}} v_j.
\]
Then we can decompose the vector space $\bM_{\sQ,\bd}(\bV)$ as
\[
   \bM_{\sQ,\bd}(\bV) \simeq 
   \Hom_{\C}\left(\widetilde{V}_i,V_i \otimes R_{d_i}\right) 
   \oplus \Hom_{\C}\left(V_i \otimes R_{d_i},\widetilde{V}_i\right) \oplus \bM_{\sQ,\bd}^{(i)}(\bV),
\]
where 
$\bM_{\sQ,\bd}^{(i)}(\bV)\coloneqq \bigoplus_{t(h),s(h) \neq i} \Hom_{\C}\left(V_{s(h)} \otimes R_{d_{s(h)}},V_{t(h)} \otimes R_{d_{t(h)}}\right) .$
Each $\bB \in \bM_{\sQ,\bd}(\bV)$
corresponds to the triple
$(B_{i \shortleftarrow},B_{\shortleftarrow i},B_{\neq i})$,
where
\begin{flalign*}
    & B_{i \shortleftarrow}\coloneqq (\sgn(h)\alpha_h^{-1}(B_h))_{t(h)=i} \in \Hom_{\C}\left(\widetilde{V}_i,V_i \otimes R_{d_i}\right),                                                                                                                                                        \\
    & B_{\shortleftarrow i}\coloneqq (\beta_{\overline{h}}^{-1}(B_ {\overline{h}}))_{t(h)=i} \in \Hom_{\C}\left(V_i \otimes R_{d_i},\widetilde{V}_i\right), \\
    & B_{\neq i} \coloneqq (B_h)_{t(h),s(h) \neq i} \in  \bM_{\sQ,\bd}^{(i)}(\bV),
\end{flalign*}
and the group $G_{d_i}(V_i)$ acts trivially on $\bM_{\sQ,\bd}^{(i)}(\bV)$. 

\begin{prop}\label{lem1}
   Let $\lambda_i$ be a unit of $R_{d_i}$. 
   
   {\rm (i)} If $\dim \widetilde{V}_i < \dim V_i$, then the set 
   $\mu_{\bV,i}^{-1}(-\lambda_i\,\unit_{V_i\otimes {R_{d_i}}})$ is empty. 

   {\rm (ii)} If $\dim \widetilde{V}_i \geq \dim V_i$, 
   % then the $G_{d_i}(V_i)$-action on $\mu_{\bV,i}^{-1}(-\lambda_i\,\unit_{V_i\otimes {R_{d_i}}})$ 
   % has a geometric quotient, and the map 
   then the map 
   \[
      \Phi_i \colon \bM_{\sQ,\bd}(\bV) \to \g_{d_i}(\widetilde{V}_i) \times \bM_{\sQ,\bd}^{(i)}(\bV); 
      \quad \bB \mapsto \left( -B_{\shortleftarrow i}^{R_{d_i}}B_{i \shortleftarrow}^{R_{d_i}}, B_{\neq i} \right)
   \]
   % where $N_i \in \gl_\C(V_i \otimes R_{d_i})$ is the multiplication by $\epsilon_{d_i}$, 
   % induces a symplectic isomorphism 
   % \[
   %    \mu_{\bV,i}^{-1}(-\lambda_i\,\unit_{V_i\otimes R_{d_i}})/G_{d_i}(V_i) \simeq  
   %    \calO \times \bM_{\sQ,\bd}^{(i)}(\bV),
   % \]
   restricts to a principal $G_{d_i}(V_i)$-bundle 
   $\mu_{\bV,i}^{-1}(-\lambda_i\,\unit_{V_i\otimes R_{d_i}}) \to \calO \times \bM_{\sQ,\bd}^{(i)}(\bV)$,
   where $\calO$ is the $G_{d_i}(\widetilde{V}_i)$-coadjoint orbit 
   consisting of elements having a matrix representation of the form  
   $\diag (\lambda_i, \dots ,\lambda_i, 0, \dots ,0)$ with $\lambda_i$ appearing 
   $\dim V_i$ times in the diagonal entries. 
   Via this map, $\calO \times \bM_{\sQ,\bd}^{(i)}(\bV)$ is a Hamiltonian reduction 
   of $\bM_{\sQ,\bd}(\bV)$.
\end{prop}

\begin{proof}
   Suppose that $\mu_{\bV,i}^{-1}(-\lambda_i\,\unit_{V_i\otimes{R_{d_i}}})$ is non-empty 
   and take any $\bB \in \mu_{\bV,i}^{-1}(-\lambda_i\,\unit_{V_i\otimes{R_{d_i}}})$.  
   Note that for any $B\in \Hom_{R_{d_h}}(V_{s(h)}\otimes R_{d_{s(h)}},V_{t(h)}\otimes R_{d_{t(h)}})$, 
   we have 
   \[
      B^{R_{d(s(h))}}=\beta_h^{-1}(B)^{R_{d(s(h))}},\quad 
      B^{R_{d(t(h))}}=\alpha_h^{-1}(B)^{R_{d(t(h))}},
   \]
   where we use the following identifications: 
   \[
      V_{t(h)} \otimes R_{d_{t(h)}}\otimes_{R_h} R_{d_{s(h)}}
      \simeq V_{\overline{h}} \otimes R_{d_{s(h)}},\quad
      V_{s(h)} \otimes R_{d_{s(h)}}\otimes_{R_{\overline{h}}} R_{d_{t(h)}}\simeq V_{h} \otimes R_{d_{t(h)}}.
   \]
   % \begin{align*}
   %    &\gamma \colon \Hom_{R_{d_{s(h)}}}\left(V_{s(h)} \otimes R_{d_{s(h)}},V_{t(h)} \otimes R_{d_{t(h)}}\otimes_{R_h} R_{d_{s(h)}}\right)\xrightarrow{\sim} \Hom_{R_{d_{s(h)}}}\left(V_{s(h)} \otimes R_{d_{s(h)}},V_{\overline{h}} \otimes R_{d_{s(h)}}\right),\\
   %    &\gamma' \colon \Hom_{R_{d_{t(h)}}}\left(V_{s(h)} \otimes R_{d_{s(h)}}\otimes_{R_{\overline{h}}} R_{d_{t(h)}},V_{t(h)} \otimes R_{d_{t(h)}}\right)\xrightarrow{\sim} \Hom_{R_{d_{t(h)}}}\left(V_{h} \otimes R_{d_{t(h)}},V_{t(h)} \otimes R_{d_{t(h)}}\right).
   % \end{align*}
Hence we have
\[
   \mu_{\bV,i}(\bB)= \sum_{\substack{h\in H, \\t(h)=i}} \sgn(h) B_h^{R_{d_i}} B_{\overline{h}}^{R_{d_i}}= \sum_{\substack{h\in H, \\t(h)=i}} \sgn(h) \alpha^{-1}(B_h)^{R_{d_i}} \beta^{-1}(B_{\overline{h}})^{R_{d_i}} = B_{i \shortleftarrow}^{R_{d_i}} B_{\shortleftarrow i}^{R_{d_i}}.
\] 
   Since $\lambda_i$ is a unit, 
   the moment map condition 
   $B_{i \shortleftarrow}^{R_{d_i}} B_{\shortleftarrow i}^{R_{d_i}} = -\lambda_i\,\unit_{V_i\otimes{R_{d_i}}}$
   implies that $B_{i \shortleftarrow}^{R_{d_i}}$ is surjective and 
   $B_{\shortleftarrow i}^{R_{d_i}}$ is injective. 
   In particular, we have $\dim \widetilde{V}_i \geq \dim V_i$. 
   Now, we embed $V_i$ into $\widetilde{V_i}$ as a vector subspace, 
   take any complement $V'_i$ (so $\widetilde{V_i} = V'_i \oplus V_i$) 
   and apply Proposition~\ref{prp:HWW} to 
   $\Theta = (0\,\unit_{V'_i \otimes R_{d_i}}) \oplus (\lambda_i\,\unit_{V_i \otimes R_{d_i}}) \in \g_{d_i}(\widetilde{V}_i)$
   with $l=1,\,W_0=V'_i,\,W_1=V_i,\,\theta_0=0,\,\theta_1=\lambda_i$. 
   Then $\calO_\Theta =\calO$ and the map 
   \[
      M \coloneqq \Hom_{\C}\left(\widetilde{V}_i,V_i \otimes R_{d_i}\right) 
      \oplus \Hom_{\C}\left(V_i \otimes R_{d_i},\widetilde{V}_i\right) 
      \to \g_{d_i}(\widetilde{V}_i), \quad 
      (B_{i \shortleftarrow},B_{\shortleftarrow i}) \mapsto 
      -B_{\shortleftarrow i}^{R_{d_i}} B_{i \shortleftarrow}^{R_{d_i}}
   \]
   restricts to a principal $G_{d_i}(V_i)$-bundle 
   \[
     Z \coloneqq \{\,(B_{i \shortleftarrow},B_{\shortleftarrow i}) \in M \mid B_{i \shortleftarrow}^{R_{d_i}} B_{\shortleftarrow i}^{R_{d_i}} 
     = -\lambda_i\,\unit_{V_i \otimes R_{d_i}} \,\} 
     \to \calO,
   \]
   via which $\calO$ is a Hamiltonian reduction of $M$. 
   Since $\mu_{\bV,i}^{-1}(-\lambda_i\,\unit_{V_i\otimes{R_{d_i}}}) = Z \times \bM_{\sQ,\bd}^{(i)}(\bV)$ 
   and $G_{d_i}(V_i)$ acts trivially on $\bM_{\sQ,\bd}^{(i)}(\bV)$, 
   the assertion follows.
%    Under the notation of Proposition~\ref{prp:HWW}, we have 
%  \[
%    \tilde{\mu}_{\bV,1}(B_{i \shortleftarrow},B_{\shortleftarrow i})=B_{i \shortleftarrow}^{R_{d_i}} B_{\shortleftarrow i}^{R_{d_i}} = -\lambda_i\,\unit_{V_i\otimes{R_{d_i}}}.
%  \]
%  Thus the assertion follows.
\end{proof}

By Proposition~\ref{lem1}, the level set 
$\mu_{\bV,i}^{-1}(-\lambda_i\,\unit_{V_i\otimes{R_{d_i}}})$ is non-empty if and only if
\[
   v_i \leq \dim \widetilde{V_i}=2v_i -\sum_{j}c_{ij}v_j,
\]
which is equivalent to $s_i(\bv) \in \Z_{\geq 0}^I$ 
as the $i$-th component of $s_i(\bv)$ is equal to $\dim \widetilde{V}_i - \dim V_i$.
We assume this condition, because otherwise
$\QS_{\sQ,\bd}(\bmlam,\bv)$ and
$\QS_{\sQ,\bd}(r_i(\bmlam),s_i(\bv))$
are both empty.
As in the proof of Proposition~\ref{lem1}, 
we embed $V_i$ into $\widetilde{V_i}$ as a vector subspace 
and take any complement $V'_i$.
By Proposition~\ref{lem1}, we have an isomorphism
\[
   \mu_{\bV,i}^{-1}(-\lambda_i\,\unit_{V_i\otimes R_{d_i}})/G_{d_i}(V_i) 
   \xrightarrow{\sim} \calO\times \bM_{\sQ,\bd}^{(i)}(\bV),
\]
where $\calO$ is the $G_{d_i}(\widetilde{V}_i)$-coadjoint orbit of
\[
   \Lambda = \begin{pmatrix} \lambda_i\,\unit_{V_i\otimes R_{d_i}} & 0 \\ 0& 0\,\unit_{V'_i\otimes R_{d_i}} \end{pmatrix}.
\]
We define an $I$-graded $\C$-vector space $\bV'$ by
\[
   \bV'=\bigoplus_{j\in I}V'_j, \quad V'_j=\begin{cases}V'_i & \text{if}  \ j=i, \\V_j& \text{if} \ j \neq i.
   \end{cases}
\]
Then $\dim \bV'=s_i(\bv)$.
By replacing $\bV$ and $\lambda_i$ with $\bV'$ and $-\lambda_i$, respectively  
in Proposition~\ref{lem1}, 
we also have an isomorphism
\[
   \mu_{\bV',i}^{-1}(\lambda_i\,\unit_{V'_i\otimes R_{d_i}})/G_{d_i}(V'_i) \xrightarrow{\sim} 
   \calO'\times \bM_{\sQ,\bd}^{(i)}(\bV'),
\]
where $\calO'$ is the 
$G_{d_i}(\widetilde{V}_i)$-coadjoint orbit of 
\[
   \Lambda'=\begin{pmatrix} 0\,\unit_{V_i\otimes R_{d_i}} & 0 \\ 0& -\lambda_i\,\unit_{V'_i\otimes R_{d_i}} \end{pmatrix}=\Lambda -\lambda_i\,\unit_{\widetilde{V}_i}.
\]
Note that $\bM_{\sQ,\bd}^{(i)}(\bV') = \bM_{\sQ,\bd}^{(i)}(\bV)$.
Therefore, the scalar shift 
$\calO\xrightarrow{\sim}\calO'$ induces an isomorphism
\[
   \widetilde{\mathcal{F}}_i \colon 
   \mu_{\bV,i}^{-1}(-\lambda_i\,\unit_{V_i\otimes R_{d_i}})/G_{d_i}(V_i)
   \xrightarrow{\sim} \mu_{\bV',i}^{-1}(\lambda_i\,\unit_{V'_i\otimes R_{d_i}})/G_{d_i}(V'_i).
\]
For $\bB \in \mu_{\bV,i}^{-1}(-\lambda_i\,\unit_{V_i\otimes R_{d_i}})$, 
take $\bB' \in \mu_{\bV',i}^{-1}(\lambda_i\,\unit_{V'_i\otimes R_{d_i}})$ so that 
$\widetilde{\mathcal{F}}_i[\bB]=[\bB']$. 

\begin{lem}\label{key}
   If $\mu_{\bV}(\bB)=-\bmlam\,\unit_{\bV},$ then 
   $\mu_{\bV'}(\bB')=-r_i(\bmlam)\,\unit_{\bV'}.$
\end{lem}

\begin{proof}
   By the definition, 
   $\Phi_i(\bB')$ is equal to $\Phi_i(\bB) -\lambda_i\,\unit_{\widetilde{V}_i}$.
   Thus we have 
   \[
      \sum_{k=0}^{d_i-1}B'_{\shortleftarrow i}(N'_i)^{k}B'_{i\shortleftarrow}\epsilon_i^{d_i-k-1} 
      =\sum_{k=0}^{d_i-1}B_{\shortleftarrow i}N_i^{k}B_{i\shortleftarrow}\epsilon_i^{d_i-k-1}
      +\lambda_i\,\unit_{\widetilde{V}_i},
   \]
   where $N_i \in \gl_\C(V_i \otimes R_{d_i})$, $N'_i \in \gl_\C(V'_i \otimes R_{d_i})$ 
   are the multiplication by $\epsilon_i$. 
   This implies that for any arrow $h$ with $t(h)=i$, 
   the following equality holds:
   \[
      \sum_{k=0}^{d_i-1}\sgn(h)\beta_{\overline{h}}^{-1}(B'_{\overline{h}})(N'_i)^{k}\alpha_h^{-1}(B'_h)\epsilon_i^{d_i-k-1}=\sum_{k=0}^{d_i-1}\sgn(h)\beta_{\overline{h}}^{-1}(B_{\overline{h}})N_i^{k}\alpha_h^{-1}(B_h)\epsilon_i^{d_i-k-1}
      +\lambda_i\,\unit_{V_h}.
   \]
   For all $l=0,\ldots,d_h-1$, comparing the coefficient of 
   $\epsilon_i^{d_i-f_hl-1}$ in the above equality yields
   \[
      \sgn(h)\beta_{\overline{h}}^{-1}(B'_{\overline{h}})(N'_h)^{l}\alpha_h^{-1}(B'_h)
      =\sgn(h)\beta_{\overline{h}}^{-1}(B_{\overline{h}})N_h^{l}\alpha_h^{-1}(B_h)+\lambda_{i,(d_i-f_hl-1)}\unit_{V_h},
   \]
   where $N_h =N_i^{f_h}$ and $N'_h = (N'_i)^{f_h}$.
   On the other hand, for  
   $B \in \Hom_{\C}(V_h,V_{t(h)}\otimes R_{d_{t(h)}})$ and 
   $\overline{B} \in \Hom_\C(V_{s(h)}\otimes R_{d_{s(h)}}, V_{\overline{h}})$
   we have
   \[
      \beta_{\overline{h}}(\overline{B})\alpha _h(B)
      %=\epsilon_h^{d_h}\overline{B}(\epsilon_h \unit - N_{h})^{-1}B 
      =\sum_{l=0}^{d_h-1}\ov{B} N_h^l B \epsilon_h^{d_h-l-1}
   \]
   as elements of $\End_{R_{d_h}}(V_{s(h)} \otimes R_{d_{s(h)}}) = \gl_\C(V_h) \otimes R_{d_h}$.
   Thus 
   \begin{align*}
      \sgn(h)B'_{\overline{h}}B'_{h} 
      &= \sgn(h) \sum_{l=0}^{d_h-1} \beta_{\overline{h}}^{-1}(B'_{\overline{h}})(N'_h)^{l}\alpha_h^{-1}(B'_h) \epsilon_h^{d_h -l-1} \\
      &= \sgn(h) \sum_{l=0}^{d_h-1} \beta_{\overline{h}}^{-1}(B_{\overline{h}})N_h^{l}\alpha_h^{-1}(B_h) \epsilon_h^{d_h -l-1} 
      + \sum_{l=0}^{d_h-1} \lambda_{i,(d_i-f_hl-1)} \epsilon_h^{d_h -l-1} \unit_{V_h} \\
      &= \sgn(h)B_{\overline{h}}B_h + \sum_{l=0}^{d_h-1}\lambda_{i,(d_i-f_hl-1)}\epsilon_h^{d_h-l-1}\unit_{V_h}.
   \end{align*}
   Replacing $h$ with $\overline{h}$, we also obtain 
   \[
      \sgn(h)B'_{h}B'_{\overline{h}}=\sgn(h)B_hB_{\overline{h}} - \sum_{l=0}^{d_h-1}\lambda_{i,(d_i-f_{\overline{h}}l-1)}\epsilon_h^{d_h-l-1}\unit_{V_{\overline{h}}}
   \]
   for any arrow $h$ with $s(h)=i$.
   Note that 
   \[
      \pr_{d_h,d_{t(h)}}\left(\epsilon_{h}^{d_h-l-1}\unit_{V_{\overline{h}}}\right)=\sum_{k=0}^{f_h-1}N_{t(h)}^k\left(\epsilon_h^{d_h-l-1}\unit_{V_{\overline{h}}}\right)N_{t(h)}^{f_h-k-1}\unit_{V_{t(h)}}=f_{h}\epsilon_{t(h)}^{d_{t(h)}-f_hl-1}\unit_{V_{t(h)}}
   \]
  by Lemma \ref{lem:pr}.
  Thus 
   \[
      \pr_{d_h,d_{t(h)}}\left(\sgn(h)B'_hB'_{\overline{h}}\right)=\pr_{d_h,d_{t(h)}}\left(\sgn(h)B_hB_{\overline{h}}\right) -\sum_{l=0}^{d_h-1}\lambda_{i,(d_i-f_{\overline{h}}l-1)}f_{h}\epsilon_{t(h)}^{d_{t(h)}-f_hl-1}\unit_{V_{t(h)}}.
   \]
   On the other hand, since $B'_h=B_h$ whenever $t(h),s(h)\neq i$,
   we have
   \[
      \pr_{d_h,d_{t(h)}}(B'_hB'_{\overline{h}})=\pr_{d_h,d_{t(h)}}(B_hB_{\overline{h}}).
   \]
   Thus, for all $j \neq i$, we obtain
   \begin{align*}
      \mu_{\bV',j}(\bB') &=\mu_{\bV,j}(\bB)
      -\sum_{\substack{h\in H \\s(h)=i,t(h)=j}}\sum_{l=0}^{d_h-1}\lambda_{i,(d_i-f_{\overline{h}}l-1)}f_{h}\epsilon_j^{d_j-f_{h}l-1}\unit_{V_{t(h)}} \\
      &=\mu_{\bV,j}(\bB)
      +\sum_{l=0}^{d_{ij}-1}c_{ij}\lambda_{i,(d_i-\frac{d_i}{d_{ij}}l-1)}\epsilon_j^{d_j-\frac{d_j}{d_{ij}}l-1}\unit_{V_j},
   \end{align*}
   whence the result.

\end{proof}
\begin{proof}[Proof of Theorem $\ref{thm:reflection}$]
   Restricting the principal $G_{d_i}(V_i)$-bundle 
   $\mu_{\bV,i}^{-1}(-\lambda_i\,\unit_{V_i\otimes R_{d_i}}) \xrightarrow{\Phi_i} \calO \times \bM_{\sQ,\bd}^{(i)}(\bV)$
   given in Proposition~\ref{lem1} to the closed subset $\mu_\bV^{-1}(-\bmlam\,\unit_\bV)$, 
   we obtain a geometric quotient 
   \[
     Z \coloneqq \Phi_i \left( \mu_\bV^{-1}(-\bmlam\,\unit_\bV) \right) 
     = \mu_\bV^{-1}(-\bmlam\,\unit_\bV)/G_{d_i}(V_i).
   \]
   Similarly we obtain a geometric quotient 
   $Z' \coloneqq \mu_{\bV'}^{-1}(-r_i(\bmlam) \unit_{\bV'})/G_{d_i}(V_i')$ 
   of $\mu_{\bV'}^{-1}(-r_i(\bmlam) \unit_{\bV'})$.
   By Lemma~\ref{key}, the isomorphism $\widetilde{\mathcal{F}}_i$ restricts to 
   an isomorphism $Z \xrightarrow{\sim} Z'$,
   which is $\prod_{j\neq i}G_{d_j}(V_j)$-equivariant.
   Thus it induces an isomorphism of $\C$-algebras
   \[
      \C[Z]^{\prod_{j\neq i}G_{d_j}(V_j)}\xrightarrow{\sim} 
      \C[Z']^{\prod_{j\neq i}G_{d_j}(V_j)}.
   \]
   Since $\C[Z] = \C[\mu_\bV^{-1}(-\bmlam\,\unit_\bV)]^{G_{d_i}(V_i)}$, 
   $\C[Z']=\C[\mu_{\bV'}^{-1}(-r_i(\bmlam) \unit_{\bV'})]^{G_{d_i}(V_i')}$, 
   we obtain an isomorphism of $\C$-schemes
   \[
      \QS_{\sQ,\bd}(\bmlam,\bv) \xrightarrow{\sim} \QS_{\sQ,\bd}(r_i(\bmlam),s_i(\bv)).
   \]
\end{proof}

\section{Regularization}\label{section:normalization}

In this section, we generalize \cite[Theorem~5.8]{YD1} 
using a result of Hausel--Wong--Wyss.

\subsection{Shifting trick}

In this subsection we fix a finite dimensional $\C$-vector space $V$ 
together with a positive integer $d$, 
and recall a sort of ``shifting trick'' found by Boalch~\cite{Boa01}  
relating to the $G_d(V)$-coadjoint orbits considered in Section~\ref{subsec:HWW}. 
For simplicity, we put $G \coloneqq \GL_\C(V)$ and $\g \coloneqq \gl_\C(V)$. 

Let $B_d(V)$ be the kernel of the homomorphism 
\[
   G_d(V) \to G; \quad 
   g=\sum_{k=0}^{d-1} g_k \epsilon^k \mapsto g_0,
\] 
and $\mathfrak{b}_d(V)$ be its Lie algebra. 
Then we have a direct sum decomposition 
\[
   \g_d(V) = \g \oplus \mathfrak{b}_d(V).
\]
Taking dual via the pairing on $\g_d(V)$, 
we also have a decomposition 
\begin{equation}\label{eq:dec}
   \g_d(V) = \epsilon^{d-1}\g_d(V) \oplus \mathfrak{b}_d^*(V),
\end{equation}
where 
\[
   \mathfrak{b}_d^*(V) \coloneqq 
   \sum_{k=0}^{d-2} \gl_\C(V) \epsilon^k \simeq 
   \g_d(V)/\epsilon^{d-1}\g_d(V).
\]
It may be regarded as the dual space of $\mathfrak{b}_d(V)$, 
and the coadjoint action of $g \in B_d(V)$ is given by
\[
   g\cdot \eta  = g \eta g^{-1} \mod \epsilon^{d-1}\g_d(V).
\]
According to the decomposition~\eqref{eq:dec}, 
we can decompose $A = \sum_{k=0}^{d-1} A_k \epsilon^k \in \g_d(V)$ as 
\[
   A = \epsilon^{d-1} A_{d-1} + A^0, \quad 
   A^0 \in \mathfrak{b}_d^*(V).
\] 

Now take any direct sum decomposition $V = \bigoplus_{i=0}^l W_i$ 
and elements $\theta_0,\theta_1, \dots ,\theta_l \in R_d$ so that 
$\theta_i - \theta_j$ is a unit whenever $i \neq j$. 
Put 
\[
   \Theta \coloneqq \bigoplus_{i=0}^l \theta_i\,\unit_{W_i \otimes R_d} \in \g_d(V),
\]
and consider the $G_d(V)$-coadjoint orbit $\calO_\Theta$ 
of $\Theta$ as in Section~\ref{subsec:HWW}.  
Let $\ccalO_\Theta \subset \mathfrak{b}_d^*(V)$ 
be the $B_d(V)$-coadjoint orbit of $\Theta^0$ 
and put 
\[
   G_\Theta \coloneqq \prod_{i=0}^l \GL_\C(W_i) \subset G,
\]
whose Lie algebra is 
$\g_\Theta \coloneqq \bigoplus_{i=0}^l \gl_\C(W_i) \subset \g$. 
Using the trace pairing we identify
the dual space $\g_\Theta^*$ with $\g_\Theta$.
Since $gbg^{-1} \in B_d(V)$ and $g\Theta^0 g^{-1}=\Theta^0$ 
for all $g \in G_\Theta$ and $b \in B_d(V)$, 
we see that the orbit $\ccalO_\Theta$
is invariant under the conjugation by $G_\Theta$.

\begin{prop}\label{prp:HY}
   There exists an $G_\Theta$-equivariant symplectic isomorphism 
   \[
      \ccalO_\Theta \xrightarrow{\sim} 
      \bigoplus_{i<j} \Hom_\C(W_i,W_j)^{\oplus (d-2)} \oplus 
      \bigoplus_{i>j} \Hom_\C(W_i,W_j)^{\oplus (d-2)}
   \]
   sending $\Theta^0$ to the origin.
\end{prop}

\begin{proof}
   This is a special case of \cite[Corollary~3.9]{HY}.
\end{proof}

In particular, $\ccalO_\Theta$ is affine and 
the $G_\Theta$-action on $\ccalO_\Theta$ admits a moment map 
$\mu_{\ccalO_\Theta} \colon \ccalO_\Theta \to \g_\Theta^* \simeq \g_\Theta$ 
with $\mu_{\ccalO_\Theta}(\Theta^0)=0$. 

We let $G_\Theta$ act on the cotangent bundle $T^*G$ by the left translation 
and consider the diagonal action on the 
product $T^*G \times \ccalO_\Theta$,
which has a moment map 
\[
   \mu_{T^*G \times \ccalO_\Theta} 
   \colon T^*G \times \ccalO_\Theta \to \g_\Theta^*; \quad 
   (g,R,B) \mapsto -\pr_{\g_\Theta}(gRg^{-1}) + \mu_{\ccalO_\Theta}(B),
\] 
where $T^*G$ is identified with $G \times \g$ via the left translation 
and $\pr_{\g_\Theta} \colon \g \to \g_\Theta$ is the transpose of the 
inclusion $\g_\Theta \hookrightarrow \g$. 

Note that $\Theta_{d-1}$ lies in $\g_\Theta^{G_\Theta}$.

\begin{prop}
   The $G_\Theta$-action on the level set 
   $\mu_{T^*G \times \ccalO_\Theta}^{-1}(-\Theta_{d-1})$ 
   is free and the affine quotient $\mu_{T^*G \times \ccalO_\Theta}^{-1}(-\Theta_{d-1})/G_\Theta$  
   is a geometric quotient. Moreover, the map 
   \[
      \mu_{T^*G \times \ccalO_\Theta}^{-1}(-\Theta_{d-1}) \to \g_d(V);
      \quad (g,R,B) \mapsto \epsilon^{d-1} R + g^{-1}Bg 
   \]
   induces a symplectic isomorphism 
   \[
      \mu_{T^*G \times \ccalO_\Theta}^{-1}(-\Theta_{d-1})/G_\Theta 
      \xrightarrow{\sim} \calO_\Theta.
   \]   
\end{prop}

\begin{proof}
   The $G_\Theta$-action on 
   $\mu_{T^*G \times \ccalO_\Theta}^{-1}(-\Theta_{d-1})$ 
   is free as $G_\Theta$ acts freely on $T^*G$. 
   Hence all the $G_\Theta$-orbits have equal dimension, and hence are closed. 
   Thus \cite[Theorem~4.10]{VP} implies that the affine quotient 
   $\mu_{T^*G \times \ccalO_\Theta}^{-1}(-\Theta_{d-1})/G_\Theta$
   is a geometric quotient. 
   For the rest assertions, see \cite[Propositions~2.6, 2.12]{HY}.
\end{proof}

\begin{cor}\label{crl:affine}
   The orbit $\calO_\Theta$ is affine.
\end{cor}

\begin{cor}\label{lem3}
   Let $M$ be a non-singular affine symplectic variety 
   acted on by $G$ in Hamiltonian fashion with
   moment map $\mu_M\colon M \to \g$.
   Then for each $\zeta \in \C$, the map
   \[
      \psi \colon \ccalO_\Theta \times M \to \g_d(V) \times M; \quad
      (B,x) \mapsto (B-\epsilon^{d-1} \mu_M (x) - \epsilon^{d-1} \zeta\,\unit_V,x)
   \]
   induces an isomorphism between affine quotients
   \[
      \mu_{\ccalO_\Theta \times M}^{-1}(-\Theta_{d-1} -\zeta\,\unit_V) /G_\Theta 
      \xrightarrow{\sim}
      \mu_{\calO_\Theta \times M}^{-1}(-\zeta\,\unit_V)/G,
   \]
   where $\mu_{\ccalO_\Theta\times M}$ and 
   $\mu_{\calO_\Theta\times M}$ are the moment maps 
   \[
      \mu_{\ccalO_\Theta\times M}(B,x)=\mu_{\ccalO_\Theta}(B)+\pr_{\g_\Theta}(\mu_M(x)), \quad 
      \mu_{\calO_\Theta\times M}(A,x) = A_{d-1} + \mu_M(x) \quad 
      (B \in \ccalO_\Theta,\ A \in \calO_\Theta,\ x \in M).
   \]
\end{cor}

\begin{proof}
   By the above proposition, the Hamiltonian reduction of 
   $\calO_\Theta \times M$ by the $G$-action at level $-\zeta\,\unit_V$ 
   is isomorphic to that of 
   $T^*G \times \ccalO_\Theta \times M$ 
   by the $G_\Theta \times G$-action 
   \[
      (u,v) \colon (g,R,B,x) \mapsto (ugv^{-1},vRv^{-1},uBu^{-1},v \cdot x),
      \quad (u,v) \in G_\Theta \times G
   \]
   at level $(-\Theta_{d-1},-\zeta\,\unit_V)$. 
   If we first perform the Hamiltonian reduction by $G$, 
   then the result is isomorphic to $\ccalO_\Theta \times M$ 
   via the map
   \[
      \ccalO_\Theta \times M \to T^*G \times \ccalO_\Theta \times M;
      \quad 
      (B,x) \mapsto (\unit_V,-\mu_M(x)-\zeta\,\unit_V,B,x),
   \]
   with the induced $G_\Theta$-moment map equal to $\mu_{\ccalO_\Theta\times M} + \zeta\,\unit_V$.
   Thus performing further the Hamiltonian reduction by $G_\Theta$, 
   we obtain a desired isomorphism 
   \[
    \mu_{\ccalO_\Theta \times M}^{-1}(-\Theta_{d-1} -\zeta\,\unit_V) /G_\Theta 
    \xrightarrow{\sim} \mu_{\calO_\Theta \times M}^{-1}(-\zeta\,\unit_V)/G,
    \] 
   which is explicitly given by
   $(B,x) \mapsto (-\epsilon^{d-1}(\mu_M(x)+\zeta\,\unit_V)+B,x)$.
\end{proof}

\subsection{Irregular legs and regularization}

Let $\sQ=(I,\Omega,s,t)$ be a quiver with multiplicities $\bd$. For integers $i<j$ we put $[i,j] \coloneqq \{i,i+1, \dots ,j\}$.

\begin{defin}
   $(\sQ,\bd)$ is said to have an \emph{irregular leg} 
   if there exists a sequence of pairwise distinct vertices such that, 
   if we denote it by $0,1, \dots ,l$, 
   then $l>0$ and the following hold:   
   \begin{enumerate}
      \item vertices $i,j$ in $[0,l]$ are connected by exactly one arrow if $\lvert i-j \rvert =1$, and otherwise no arrow connects them; 
      \item no arrow connects any $i \in I \setminus [0,l]$ and $j \in [1,l]$;
      \item $d_0=1$ and $d_i=d$ ($i=1,2, \dots ,l$) for some integer $d>1$.
   \end{enumerate}
\end{defin}

In what follows we consider such a quiver with multiplicities, 
and for simplicity, assume that the arrows 
connecting $0,1, \dots ,l$ are oriented as $0 \to 1 \to \cdots \to l$. 
We denote by $\sQ_\leg = ([1,l],\Omega_\leg,s,t)$ 
the subquiver $1 \to 2 \to \cdots \to l$ and call it 
the \emph{irregular leg} of $(\sQ,\bd)$ with \emph{base} $0$. 

\begin{defin}
   Let $\check{\sQ}=(I,\check{\Omega},s,t)$ be the quiver  
   obtained from $\sQ$ by the following procedure:
   \begin{enumerate}
      \item first, delete the $l$ arrows $0 \to 1 \to \cdots \to l$; then 
      \item for each arrow $h$ with $t(h)=0$ and each $i \in [1,l]$, add an arrow from $s(h)$ to $i$;
      \item for each arrow $h$ with $s(h)=0$ and each $i \in [1,l]$, add an arrow from $i$ to $t(h)$;
      \item finally, for each pair $i<j$ in $[0,l]$, add $(d-2)$ arrows from $i$ to $j$.
   \end{enumerate}
   Also, define $\check{\bd}=(\check{d}_i)$ by 
   \[
      \check{d}_i = 
      \begin{cases}
         1 & (i \in [1,l]), \\
         d_i & (i \in I \setminus [1,l]).
      \end{cases}
   \] 
   We call $(\check{\sQ},\check{\bd})$ 
   the \emph{regularization} of $(\sQ,\bd)$ along the irregular leg $\sQ_\leg$.
   In Example \ref{example} we will give some examples of regularization.
\end{defin}

\begin{rem}
   When $l=1$, the regularization is the same as 
   the \emph{normalization} introduced by the second named author in 
   \cite{YD1}.
\end{rem}

We define a map $R_{\bd} \times \Z^I \to R_{\check{\bd}} \times \Z^I$ as follows.
For $\bv=(v_i) \in \Z^I$, define $\check{\bv}=(\check{v}_i) \in \Z^I$ by 
\[
   \check{v}_i = 
   \begin{cases}
      v_i - v_{i+1} & (i \in [0,l-1]), \\
      v_i & (\text{otherwise}).
   \end{cases}
\]
Also, for $\bmlam=(\lambda_i) \in R_{\bd}$, 
define $\check{\bmlam} = (\check{\lambda}_i) \in R_{\check{\bd}}$ by 
\begin{equation}\label{eq:lambda}
   \check{\lambda}_i = 
   \begin{cases}
      \lambda_0 & (i=0), \\
      \lambda_0 + \sum_{j=1}^i \lambda_{j,d-1} & (i \in [1,l]), \\
      \lambda_i & (\text{otherwise}).
   \end{cases}
\end{equation}

The following theorem generalizes \cite[Theorem~5.8]{YD1}.
\begin{thm}\label{main}
   Let $(\sQ,\bd)$ be a quiver with multiplicities 
   having an irregular leg $\sQ_\leg$ as above, 
   and let $(\check{\sQ},\check{\bd})$ be 
   the regularization of $(\sQ,\bd)$ along $\sQ_\leg$. 
   Take a pair $(\bmlam,\bv) \in R_{\bd} \times \Z_{\geq 0}^I$ satisfying 
   the following conditions:
   \begin{enumerate}
      \item $\check{v}_i \geq 0$ for all $i \in [0,l-1]$;
      \item $\lambda_i + \lambda_{i+1} + \cdots + \lambda_j \in R_d^\times$ for all pairs $i \leq j$ in $[1,l]$.
   \end{enumerate}
   Then 
   $\QS_{\sQ,\bd}(\bmlam,\bv)$ and 
   $\QS_{\check{\sQ},\check{\mathbf{d}}}(\check{\bmlam},\check{\bv})$ 
   are isomorphic.
\end{thm}

\begin{proof}
   By the first condition on $(\bmlam,\bv)$, 
   the sequence $v_0,v_1, \dots ,v_l$ is non-increasing. 
   We take $I$-graded $\C$-vector spaces $\bV, \check{\bV}$ 
   so that $\dim \bV = \bv$, $V_0 \supset V_1 \supset \cdots \supset V_l$, and  
   \[
      V_i = 
      \begin{cases}
         \check{V}_i \oplus V_{i+1} & (i \in [0,l-1]), \\
         \check{V}_i & (\text{otherwise}).
      \end{cases} 
   \]
   Then $\dim \check{\bV} = \check{\bv}$ and $V_0 = \bigoplus_{i=0}^l \check{V}_i$. 

   In what follows, for a subset $L \subset I$, 
   the suffix $L$ means the restriction of the index set to $L$; for instance,  
   \[
      \bV_L = \bigoplus_{i \in L} V_i, \quad 
      \bd_L = (d_i)_{i \in L}, \quad 
      \bmlam_L = (\lambda_i)_{i \in L}, \quad 
      \mu_{\bV,L} = (\mu_{\bV,i})_{i \in L}.   
   \]
   Let $\sQ_J$ be the full subquiver of $\sQ$ with vertex set $J \coloneqq I \setminus [1,l]$. 
   Then 
   \[
      \bM_{\sQ,\bd}(\bV) = \bM_{\sQ_J,\bd_J}(\bV_J) \oplus 
      \bM_{\sQ_\leg,\bd_{[1,l]}}(\bV_{[1,l]}) \oplus 
      \Hom_\C(V_0,V_1 \otimes R_d) \oplus \Hom_\C(V_1 \otimes R_d,V_0).
   \]
   Also, let $\sQ_K$ be the full subquiver of $\sQ$ with vertex set $K \coloneqq J \setminus \{0\}$. 
   Then $\sQ_K$ is also a subquiver of $\check{\sQ}$ and 
   $\check{d}_i = d_i$, $\check{V}_i = V_i$ for all $i \in K$. Hence 
   \begin{align*}
      \bM_{\check{\sQ},\check{\bd}}(\check{\bV}) = 
      \bM_{\sQ_K,\bd_K}(\bV_K) 
      &\oplus \bigoplus_{i=0}^l 
      \left( \bigoplus_{\substack{t(h) \in K \\ s(h)=i}} \Hom_\C(\check{V}_i,\check{V}_{t(h)} \otimes R_{\check{d}_{t(h)}}) 
      \oplus \bigoplus_{\substack{s(h) \in K \\ t(h)=i}} \Hom_\C(\check{V}_{s(h)} \otimes R_{\check{d}_{s(h)}},\check{V}_i) \right) \\
      &\oplus \bigoplus_{i,j \in [0,l];\, i \neq j} \Hom_\C(\check{V}_i,\check{V}_j)^{\oplus (d-2)}.
   \end{align*}
   By the definition of $\check{\sQ}$ 
   and the equality $V_0 = \bigoplus_{i=0}^l \check{V}_i$, 
   we obtain a canonical isomorphism 
   \[
      \bM_{\check{\sQ},\check{\bd}}(\check{\bV}) \simeq 
      \bM_{\sQ_J,\bd_J}(\bV_J) \oplus 
      \bigoplus_{i,j \in [0,l];\, i \neq j} \Hom_\C(\check{V}_i,\check{V}_j)^{\oplus (d-2)}.
   \]
   Now define $\Theta \in \g_d(V_0)$ by 
   \[
      \Theta = \bigoplus_{i=0}^l \theta_i\,\unit_{\check{V}_i}, \quad 
      \theta_i = 
      \begin{cases}
         0 & (i=0), \\
         \lambda_1 + \cdots + \lambda_i & (i>0).
      \end{cases}
   \]
   Then $\theta_i - \theta_{i-1} = \lambda_i$ for $i \in [1,l]$ and 
   \[
      G_\Theta = \prod_{i=0}^l \GL_\C(\check{V}_i) 
      = G_{\check{\bd}_{[0,l]}}(\check{\bV}_{[0,l]}).
   \] 
   The second condition on $(\bmlam,\bv)$ implies that 
   $\theta_i - \theta_j \in R_d^\times$ whenever $i \neq j$. 
   Therefore Proposition~\ref{prp:HY} implies that there exists an isomorphism 
   \[
      \bM_{\check{\sQ},\check{\bd}}(\check{\bV}) \simeq 
      \bM_{\sQ_J,\bd_J}(\bV_J) \times \check{\calO}_\Theta.
   \]
   On the other hand, Proposition~\ref{prp:HWW} implies that 
   the $G_{\bd_{[1,l]}}(\bV_{[1,l]})$-action on 
   $\mu_{\bV,[1,l]}^{-1}(-\bmlam_{[1,l]}\,\unit_{\bV_{[1,l]}})$ 
   has a geometric quotient isomorphic to 
   the affine variety $\bM_{\sQ_J,\bd_J}(\bV_J) \times \calO_\Theta$. 
   Therefore Corollary~\ref{lem3} shows that there exists an isomorphism 
   between affine varieties 
   \[
      \mu_{\bV,[0,l]}^{-1}(-\bmlam_{[0,l]}\,\unit_{\bV_{[0,l]}})/G_{\bd_{[0,l]}}(\bV_{[0,l]}) 
      \simeq \mu_{\check{\bV},[0,l]}^{-1}(-\check{\bmlam}_{[0,l]}\,\unit_{\check{\bV}_{[0,l]}})/G_{\check{\bd}_{[0,l]}}(\check{\bV}_{[0,l]}).
   \]
   Taking the affine quotients (as schemes) 
   of the level sets of the $G_{\bd_K}(\bV_K)$-moment maps on both sides, 
   we obtain a desired isomorphism 
   $\QS_{\sQ,\bd}(\bmlam,\bv) \simeq 
   \QS_{\check{\sQ},\check{\mathbf{d}}}(\check{\bmlam},\check{\bv})$.   
\end{proof}

The following corollary is useful.

\begin{cor}\label{crl:main}
   Let $(\sQ,\bd)$ be a quiver with multiplicities 
   having an irregular leg $\sQ_\leg$ as above with $l=1$, 
   and let $(\check{\sQ},\check{\bd})$ be 
   the regularization of $(\sQ,\bd)$ along $\sQ_\leg$. 
   Take a pair $(\bmlam,\bv) \in R_{\bd} \times \Z_{\geq 0}^I$ 
   so that $\lambda_1 \in R_d^\times$.
   Then 
   $\QS_{\sQ,\bd}(\bmlam,\bv)$ and 
   $\QS_{\check{\sQ},\check{\mathbf{d}}}(\check{\bmlam},\check{\bv})$ 
   are isomorphic.
\end{cor}

\begin{proof}
   If $\check{v}_0 < 0$, then Proposition~\ref{lem1} implies that 
   both $\QS_{\sQ,\bd}(\bmlam,\bv)$ and 
   $\QS_{\check{\sQ},\check{\mathbf{d}}}(\check{\bmlam},\check{\bv})$ 
   are empty. If $\check{v}_0 \geq 0$, then they are isomorphic by the above theorem.
\end{proof}

Using the above corollary we can show that some quiver schemes 
are algebraic varieties.

\begin{cor}
   Let $(\sQ,\bd)$ be a quiver with multiplicities 
   and put $I_{\mathrm{irr}} \coloneqq \{\, i\in I \mid d_i >1\,\}$.  
   Suppose that each $i \in I_{\mathrm{irr}}$ is an irregular leg of length one, 
   and any distinct pair $i \neq j$ in $I_{\mathrm{irr}}$ has distinct bases.
   Take a pair $(\bmlam,\bv) \in R_{\bd} \times \Z_{\geq 0}^I$ 
   so that $\lambda_i \in R_{d_i}^\times$ for any $i \in I_{\mathrm{irr}}$.
   Then 
   $\QS_{\sQ,\bd}(\bmlam,\bv)$ is a variety.
\end{cor}

\begin{proof}
   Applying Corollary~\ref{crl:main} to each $i \in I_{\mathrm{irr}}$, 
   we obtain an isomorphism from $\QS_{\sQ,\bd}(\bmlam,\bv)$ 
   to Nakajima's quiver variety.
\end{proof}

Here are some examples.
\begin{exa}\label{example}
   (i) Consider the quiver with multiplicities
   $(\sQ,\bd)$ given in \cite[Example 5.6 (i),(ii)]{YD1},
   which has the following underlying graph with multiplicities. 
   \[
      \hfill
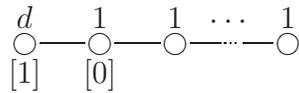
      \begin{xy}
         (0,0) *++!D{d} *\cir<4pt>{}*++!U{[1]};
         (10,0)*++!D{1} *\cir<4pt>{}*++!U{[0]};
         (20,0) *++!D{1} *\cir<4pt>{};
         (35,0)*++!D{1} *\cir<4pt>{};
         \ar@{-} (22,0);(26,0)
         \ar@{.} (26,0);(28,0)^*!U{\cdots}
         \ar@{-} (29,0);(33,0)
         \ar@{-} (2,0);(8,0)
         \ar@{-} (12,0);(18,0)

      \end{xy}
      \hfill
   \]
   Here the number of vertices is $n \geq 2$ and $[0], [1]$ 
   are labels of vertices ($d,1,1, \dots ,1$ are the multiplicities). 
   Then the regularization $(\check{\sQ},\check{\mathbf{d}})$
   has the following underlying graph with multiplicities.
   
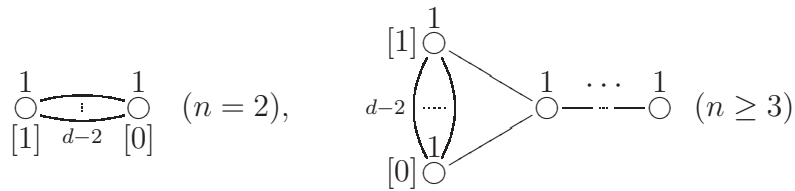
\[
      \begin{xy}
         (0,0) *++!D{1} *\cir<4pt>{}*++!U{[1]};
         (15,0)*++!D{1} *\cir<4pt>{}*++!U{[0]};
         \ar@/^2pt/ @{-} (2,1);(13,1)
         \ar@/_2pt/ @{-}(2,-1);(13,-1)_{d-2}
         \ar@{.} (7.5,0.5);(7.5,-0.5)
      \end{xy}
      \ (n=2),
      \qquad
      \begin{xy}
         (15,0) *++!D{1} *\cir<4pt>{};
         (0,8.61)   *++!R{[1]} *\cir<4pt>{}*++!D{1};
         (0,-8.61) *++!R{[0]} *\cir<4pt>{}*++!D{1};
         (30,0) *++!D{1} *\cir<4pt>{};
         \ar@{-} (13,-1);(2,-7.61)
         \ar@/^5pt/ @{-} (1,6.61);(1,-6.61)
         \ar@/_5pt/ @{-} (-1,6.61);(-1,-6.61)_{d-2}
         \ar@{-}(13,1);(2,7.61)
         \ar@{.} (-1.4,0);(1.4,0)
         \ar@{-} (17,0);(21,0)
         \ar@{.} (22,0);(23,0)^*!U{\cdots}
         \ar@{-} (24,0);(28,0)
      \end{xy}
      \ (n \geq 3)
   \]
   Since it is multiplicity-free,
   Corollary~\ref{crl:main} implies that $\QS_{\sQ,\bd}(\bmlam,\bv)$ is a variety if $\lambda_{[1]}(0) \neq 0$.

   (ii) Consider the quiver with multiplicities
   $(\sQ,\bd)$ given in \cite[Example 5.6 (iii)]{YD1},
   which has the following underlying graph with multiplicities.
   \[
      \hfill
      \begin{xy}
         (0,0) *++!D{2}*\cir<4pt>{}*++!U{[1']};
         (10,0)*++!D{1}*\cir<4pt>{}*++!U{[0']};
         (25,0) *++!D{1}*\cir<4pt>{}*++!U{[0]};
         (35,0)*++!D{2}*\cir<4pt>{}*++!U{[1]};

         \ar@{-} (2,0);(8,0)
         \ar@{-} (12,0);(16,0)
         \ar@{.} (16,0);(18,0)^*!U{\cdots}
         \ar@{-} (19,0);(23,0)
         \ar@{-} (27,0);(33,0)
      \end{xy}
      \hfill
   \]
   Here the number of vertices is $n \geq 4$. 
   Then one can perform the regularization twice, 
   and the resulting quiver with multiplicities 
   $(\check{\sQ},\check{\mathbf{d}})$
   has the following graph with multiplicities. 
   \[
      \hfill
      \begin{xy}
         (0,-7.5) *++!U{1} *\cir<4pt>{} *++!R{[1']};
         (15,-7.5) *++!U{1} *\cir<4pt>{}*++!L{[0]};
         (0,7.5) *++!D{1} *\cir<4pt>{}*++!R{[1]};
         (15,7.5) *++!D{1} *\cir<4pt>{}*++!L{[0']};
         \ar@{-} (2,-7.5);(13,-7.5)
         \ar@{-} (0,-5.5);(0,5.5)
         \ar@{-} (15,-5.5);(15,5.5)
         \ar@{-} (2,7.5);(13,7.5)
      \end{xy}
      \ (n=4),
      \qquad
      \begin{xy}
         (5,8.61) *++!D{1} *++!R{[1']} *\cir<4pt>{};
         (5,-8.61) *++!D{1} *++!R{[0']} *\cir<4pt>{};
         (10,0)   *++!D{1} *\cir<4pt>{};
         (25,0) *++!D{1} *\cir<4pt>{};
         (30,8.61) *++!D{1} *++!L{[0]} *\cir<4pt>{};
         (30,-8.61) *++!D{1} *++!L{[1]} *\cir<4pt>{};
         \ar@{-}(6,7.11);(9,2)
         \ar@{-} (6,-7.11);(9,-2)
         \ar@{-} (12,0);(16,0)
         \ar@{.} (17,0);(18,0)^*!U{\cdots}
         \ar@{-} (19,0);(23,0)
         \ar@{-} (26,2);(29,7.11)
         \ar@{-} (26,-2);(29,-7.11)
      \end{xy}
      \ (n \geq 5)
      \hfill
   \]
   Since it is multiplicity-free,
   Corollary~\ref{crl:main} implies that $\QS_{\sQ,\bd}(\bmlam,\bv)$
   is a variety if $\lambda_{[1]}(0), \lambda_{[1']}(0) \neq 0$.
\end{exa}

\subsection{Weyl groups and regularization}

Let $(\sQ,\bd)$ be a quiver with multiplicities 
having an irregular leg $\sQ_{\leg} = ([1,l],\Omega_\leg,s,t)$, 
and let $(\check{\sQ},\check{\bd})$ be the one obtained by 
the regularization of $(\sQ,\bd)$ along $\sQ_\leg$.
We denote by $\check{\bC} =2\unit - \check{\bA'}\check{\bD}$ 
the generalized Cartan matrix associated to $(\check{\sQ},\check{\bd})$,
and by $\check{\h},\check{Q},\check{\alpha}_k,\check{s}_k$
the Cartan subalgebra, the root lattice, the simple roots 
and simple reflections, of the Kac--Moody algebra with Cartan matrix $\check{\bC}$.
In this subsection we give some relationship between 
the two Weyl groups $W(\bC)$ and $W(\check{\bC})$ as in \cite[Section~5.3]{YD1}.

Define a homomorphism $\varphi\colon Q\to \check{Q}$ 
by $\bv \mapsto \check{\bv}=\bv-\sum_{i\in [0,l-1]}v_{i+1}\check{\alpha}_i$.

\begin{lem}
   If we regard $\varphi$ as an element of $\Hom_\Z(\Z^I,\Z^I)=\Z^{I \times I}$, 
   then $^t\varphi\check{\bD}\check{\bC}\varphi=\bD\bC$.
\end{lem}

\begin{proof}
   To prove it, we express the matrices in block form with respect to the decomposition of the index set 
   $I = [0,l]\sqcup K$. First, $\varphi$ is expressed as
   \[
   \varphi= 
     \begin{pmatrix}
      1&-1&0&\cdots&0&0\\
      0&\ddots&\ddots&\ddots&\vdots&\vdots\\
      0&\ddots&\ddots&\ddots&0&\vdots\\
       \vdots&\ddots&\ddots&\ddots&-1&\vdots\\
       0&\cdots&0&0&1&0\\
       0&\cdots&\cdots&0&0&\unit\\
     \end{pmatrix}.
     \]
     By the properties of $\sQ_{\leg}$ and the definition of regularization, 
     the matrices $\bD$, $\check{\bD}$, $\bA'$
     and $\check{\bA'}$ are respectively expressed as
     \[
           \bD =\begin{pmatrix}
                  \diag(1,d,\ldots,d)&0\\
                  0&\widetilde{\bD}\\
           \end{pmatrix},\quad
           \check{\bD}=\begin{pmatrix}
            \unit&0\\
            0&\widetilde{\bD}\\
           \end{pmatrix},
           \]
  \[              \bA'=\begin{pmatrix}
      0&1&0&\cdots&0&^t\mathbf{a}\\
      1&0&\frac{1}{d}&\ddots&\vdots&0\\
      0&\frac{1}{d}&\ddots&\ddots&0&\vdots\\ 
      \vdots&\ddots&\ddots&\ddots&\frac{1}{d}&0\\
      0&\cdots&0&\frac{1}{d}&0&0\\
      \mathbf{a}&0&\cdots&0&0&\widetilde{\bA'}\\
    \end{pmatrix},\quad
       \check{\bA'}=\begin{pmatrix}
      0&d-2&\cdots&d-2&^t\mathbf{a}\\
      d-2&\ddots&\ddots&\vdots&\vdots\\
      \vdots&\ddots&\ddots&d-2&\vdots\\
      d-2&\cdots&d-2&0&^t\mathbf{a}\\
      \mathbf{a}&\cdots&\cdots&\mathbf{a}&\widetilde{\bA'}\\
    \end{pmatrix},
   \] 
   where $\widetilde{\bD}$ (resp.\ $\widetilde{\bA'}$) is  
   the sub-matrix of $\bD$ (resp.\ $\bA'$) 
   obtained by restricting the index set to $K$, 
   and $\mathbf{a}=(a_{k0})_{k\in K}$. 
   Now we check the equality. We have 
   \begin{align*}
       \bD\bC=2\bD -\bD\bA'\bD
       &= \begin{pmatrix}
         2&-d&0&\cdots&0&-^t\mathbf{a}\widetilde{\bD}\\
         -d&2d&\ddots&\ddots&\vdots&0\\
         0&\ddots&\ddots&\ddots&0&\vdots\\
         \vdots&\ddots&\ddots&\ddots&-d&0\\
         0&\cdots&0&-d&2d&0\\
         -\widetilde{\bD}\mathbf{a}&0&\cdots&0&0&\widetilde{\bD}\widetilde{\bC}\\
      \end{pmatrix},
   \end{align*}
  where $\widetilde{\bC}=2\unit-\widetilde{\bA'}\widetilde{\bD}$. 
  On the other hand,
   \begin{align*}
      \check{\bD}\check{\bC}=2\check{\bD}-\check{\bD}\check{\bA'}\check{\bD}
       &=\begin{pmatrix}
         2&2-d&\cdots&2-d&-^t\mathbf{a}\widetilde{\bD}\\
         2-d&\ddots&\ddots&\vdots&\vdots\\
         \vdots&\ddots&\ddots&2-d&-^t\mathbf{a}\widetilde{\bD}\\
         2-d&\cdots&2-d&2&-^t\mathbf{a}\widetilde{\bD}\\
         -\widetilde{\bD}\mathbf{a}&\cdots&-\widetilde{\bD}\mathbf{a}&-\widetilde{\bD}\mathbf{a}&\widetilde{\bD}\widetilde{\bC}\\
       \end{pmatrix}. 
   \end{align*}
   By direct calculation, we obtain 
$^t\varphi\check{\bD}\check{\bC}\varphi
   =\bD\bC$. 
\end{proof}

The above lemma implies that $\varphi$ 
preserves the symmetric bilinear form \eqref{eq:symm}.

Let $S_{l+1}$ be the symmetric group of $[0,l]$. 
It effectively acts on $Q=\Z^I$ by permutations of coordiates.

\begin{lem}\label{lem:perm}
   $\sigma\check{s}_k\sigma^{-1}=\check{s}_{\sigma(k)}$ 
   for any $\sigma \in S_{l+1}$ and $k \in I$.
\end{lem}

\begin{proof}
   Observe that the matrices $\check{\bC}$, $\check{\bD}$ are invariant under 
   permutations of indices in $[0,l]$. 
   Hence the action of $S_{l+1}$ on $\check{Q}$ 
   preserves the symmetric bilinear form. 
   Recall that the simple reflections satisfy 
   \[
      \check{s}_i (\beta) = \beta - \frac{2(\beta,\check{\alpha}_i)}{(\check{\alpha}_i,\check{\alpha}_i)} \check{\alpha}_i 
      \quad (i \in I,\ \beta \in \check{Q}).
   \]
   For $k \in I$ and $\beta \in \check{Q}$, we thus have 
   \begin{align}
      \check{s}_{\sigma(k)}(\beta) 
      &= \beta - \frac{2(\beta,\check{\alpha}_{\sigma(k)})}{(\check{\alpha}_{\sigma(k)},\check{\alpha}_{\sigma(k)})}\check{\alpha}_{\sigma(k)} \\
      &= \beta - \frac{2(\sigma^{-1}(\beta),\check{\alpha}_k)}{(\check{\alpha}_k,\check{\alpha}_k)}\check{\alpha}_{\sigma(k)} \\ 
      &= \sigma \left( \sigma^{-1}(\beta) - \frac{2(\sigma^{-1}(\beta),\check{\alpha}_k)}{(\check{\alpha}_k,\check{\alpha}_k)}\check{\alpha}_k \right) 
      = \left( \sigma \check{s}_k \sigma^{-1} \right) (\beta).
   \end{align}
\end{proof}

If we regard $W(\check{\bC})$ and $S_{l+1}$ as subgroups of $\GL_\Z(\check{Q})$, 
then Lemma \ref{lem:perm} and the following lemma imply that $W(\check{\bC}) S_{l+1}$ is  
a semi-direct product $W(\check{\bC})\rtimes S_{l+1}$.

\begin{lem}
   $W(\check{\bC})\cap S_{l+1}=\{\unit_{\check{Q}} \}$.
\end{lem}
\begin{proof}
  Take $\sigma \in W(\check{\bC})\cap S_{l+1}$. 
  If $\sigma$ is not the identity, then 
  it sends some simple root to a negative root 
  by \cite[Lemma~3.11, (b)]{Kac}. 
  On the other hand, we have 
  $\sigma(\check{\alpha_i})=\check{\alpha}_{\sigma(i)}$ for any $i\in I$, 
  which is a contradiction.
%   Therefore \cite[Lemma~3.11 (b)]{Kac} shows that $\sigma=\unit_{\check{Q}}$.
\end{proof}

\begin{prop}
Under the isomorphism $\varphi$, 
the Weyl group $W(\bC)$ is isomorphic to the semidirect product $W(\check{\bC})\rtimes S_{l+1}$.
\end{prop}

\begin{proof}
   We calculate the subgroup 
   $\varphi W(\bC) \varphi^{-1} \subset \GL_\Z(\check{Q})$.
   Since $\varphi$ preserves the symmetric bilinear form,
   the automorphism $\varphi s_i \varphi^{-1}$ of $\check{Q}$ 
   satisfies 
   \[
      \varphi s_i \varphi^{-1}(\beta) 
      =\beta-\frac{2\left(\beta,\varphi(\alpha_i)\right)}{\left(\varphi (\alpha_i), \varphi (\alpha_i)\right)} \varphi (\alpha_i)
      \quad(i\in I,\ \beta \in \check{Q}).
   \]
   By the defintition of $\varphi$, we have
   \[
     \varphi (\alpha_i)=\begin{cases}
      \check{\alpha}_i-
      \check{\alpha}_{i-1}&(i\in[1,l]),\\
      \check{\alpha}_i & (\text{otherwise}).\end{cases} 
   \]
   It follows that $\varphi s_i \varphi^{-1} = \check{s}_i$ if $i \not\in [1,l]$. 
   For $i \in [1,l]$, a direct calculation shows 
   \[
      \left(\varphi (\alpha_i), \varphi (\alpha_i)\right) = (\alpha_i,\alpha_i) = 2d, \quad 
      (\check{\alpha}_k,\check{\alpha}_i-\check{\alpha}_{i-1})=
      \begin{cases}
         d & (k=i), \\
         -d & (k=i-1), \\
         0 & (\text{otherwise}),
      \end{cases}
   \]
   which imply that
   $\varphi s_i \varphi^{-1}(\check{\alpha}_{k})=\check{\alpha}_{\sigma_{i}(k)}$ 
   for all $k \in I$,
   where $\sigma_{i} \in S_{l+1}$ is the transposition of $i-1$ and $i$.
   Hence $\varphi s_i \varphi^{-1} = \sigma_i$.
   As a conclusion, $\varphi W(\bC) \varphi^{-1}$ is equal to the subgroup 
   generated by 
   $\sigma_i$, $i \in [1,l]$ and $\check{s}_k$, $k \not\in [1,l]$. 
   On the other hand, Lemma~\ref{lem:perm} implies that for any $i \in [1,l]$, $\check{s}_i$ is 
   conjugate to $\check{s}_0$ via the permutation $\sigma_i \cdots \sigma_2 \sigma_1$.
   Thus $\varphi W(\bC) \varphi^{-1}$ coincides with $W(\check{\bC}) S_{l+1} \simeq W(\check{\bC})\rtimes S_{l+1}$.
\end{proof}

Let $S_{l+1}$ act on $R_{\check{\bd}}$ by permutations of components.
Then it is straightforward to show (using the $S_{l+1}$-invariance of $\check{\bD}, \check{\bC}$) 
that $\sigma \check{r}_k \sigma^{-1} = \check{r}_{\sigma(k)}$ for any $\sigma \in S_{l+1}$ and $k \in I$, 
where $\check{r}_k \colon R_{\check{\bd}} \to R_{\check{\bd}}$ is the linear map 
corresponding to the simple reflection $\check{s}_k$ for the action of $W(\check{\bC})$.
Thus we obtain an action of the semi-direct product $W(\check{\bC}) \rtimes S_{l+1}$ on $R_{\check{\bd}}$.

\begin{prop}
   Let $W(\bC)$ act on $R_{\check{\bd}}\times \check{Q}$ 
   through the isomorphism $W(\bC) \simeq W(\check{\bC})\rtimes S_{l+1}$. 
   Then the map $R_{\bd}\times Q\to R_{\check{\bd}}\times \check{Q}$, $(\bmlam,\bv)\mapsto (\check{\bmlam},\check{\bv})$ is $W(\bC)$-equivariant,
   where $\check{\bmlam}$ is defined in \eqref{eq:lambda}.
\end{prop}

\begin{proof}
   Let $\psi \colon R_\bd \to R_{\check{\bd}}$ be the map $\bmlam \mapsto \check{\bmlam}$. 
   Then the transpose ${}^t \psi \colon R_{\check{\bd}} \to R_\bd$
   \[
      \check{\bmkap}=(\check{\kappa_i}) \mapsto \bmkap =(\kappa_i), \quad 
      \kappa_i = 
      \begin{cases}
         \check{\kappa}_i & (i \not\in [0,l]), \\
         \sum_{k=i}^l \check{\kappa}_k & (i \in [0,l]).
      \end{cases} 
   \]
   For the assertion it is sufficient to show that ${}^t \psi$ is equivariant with respect to the dual actions.
   For $i \in I$, let $\tilde{s}_i \colon R_\bd \to R_\bd$, $\tilde{s}'_i \colon R_{\check{\bd}} \to R_{\check{\bd}}$ be 
   the dual actions of the $i$-th simple reflection. 
   In the proof of the above lemma we checked that 
   \[
      \varphi s_i \varphi^{-1} = 
      \begin{cases}
         s_i & (i \not\in [1,l]), \\
         \sigma_i & (i \in [1,l]).
      \end{cases}
   \]
   Thus it is sufficient to show that 
   \[
      \tilde{s}_i ({}^t \psi (\check{\bmkap})) = 
      \begin{cases}
         {}^t \psi (\sigma_i (\check{\bmkap})) & (i \in [1,l]), \\
         {}^t \psi (\tilde{s}'_i (\check{\bmkap})) & (i \not\in [1,l])  
      \end{cases}
   \]
   for any $\check{\bmkap} \in R_{\check{\bd}}$.

   Fix $\check{\bmkap} \in R_{\check{\bd}}$ and put  
   $\bmkap = {}^t \psi (\check{\bmkap})$.
   First, suppose $i \in [1,l]$. 
   In the proof of Proposition~\ref{prop:reflection}, 
   we already calculated $\tilde{s}_i(\bmkap)$ as follows:  
   \[
      \tilde{s}_i(\bmkap) = \bmkap - \sum_{j \in I} c_{ij} \sum_{m=0}^{d_{ij}-1} \kappa_{j,f_{ij}m} \epsilon_{d_i}^{f_{ji}m} \bfe_i,
      \quad \bmkap =\left( \sum \kappa_{i,k} \epsilon_{d_i}^k \right).
   \]
   Since $i \in [1,l]$, we have  
   \[
      c_{ij} = 
      \begin{cases}
         2 & (j=i), \\
         -1 & (j \in [0,l],\ \lvert i-j \rvert =1), \\         
         0 & (\text{otherwise}),
      \end{cases}
   \]
   and 
   \[
      \sum_{m=0}^{d_{ij}-1} \kappa_{j,f_{ij}m} \epsilon_{d_i}^{f_{ji}m} 
      = \kappa_j 
   \]
   whenever $c_{ij} \neq 0$. Thus 
   \[
      \sum_{j \in I} c_{ij} \sum_{m=0}^{d_{ij}-1} \kappa_{j,f_{ij}m} \epsilon_{d_i}^{f_{ji}m} 
      = 2\kappa_i - \sum_{j \in [0,l];\, \lvert i-j \rvert = 1} \kappa_j = \check{\kappa}_i - \check{\kappa}_{i-1},
   \]
   and hence $\tilde{s}_i(\bmkap) = \bmkap - (\check{\kappa}_i - \check{\kappa}_{i-1}) \bfe_i$. 
   On the other hand, a direct calculation shows that 
   the $i$-th component of ${}^t \psi(\sigma_i(\check{\bmkap}))$ is equal to 
   \[
      \sum_{k=i}^l \check{\kappa}_{\sigma_i^{-1}(k)} 
      = \check{\kappa}_{i-1} + \sum_{k=i+1}^l \check{\kappa}_k 
      = \kappa_i - (\check{\kappa}_i - \check{\kappa}_{i-1}),
   \]   
   while the other components are the same as those of $\bmkap$. 
   Hence ${}^t \psi(\sigma_i(\check{\bmkap})) = \bmkap - (\check{\kappa}_i - \check{\kappa}_{i-1}) \bfe_i = \tilde{s}_i(\bmkap)$.

   Next, suppose $i \not\in [1,l]$. 
   For $j \in I$, let $\check{c}_{ij}$ be the $(i,j)$-entry of $\check{\bC}$ and 
   \[
      \check{d}_{ij} = \gcd (\check{d}_i,\check{d}_j), \quad 
      \check{f}_{ij} = \check{d}_j /\check{d}_{ij}.
   \] 
   Then we have 
   \begin{align}
      {}^t \psi (\tilde{s}'_i(\check{\bmkap})) 
      &= {}^t \psi \left( \check{\bmkap} - \sum_{j \in I} \check{c}_{ij} \sum_{m=0}^{\check{d}_{ij}-1} \check{\kappa}_{j,\check{f}_{ij}m} \epsilon_{\check{d}_i}^{\check{f}_{ji}m} \bfe_i \right) \\
      &= \bmkap - \sum_{j \in I} \check{c}_{ij} {}^t \psi \left( \sum_{m=0}^{\check{d}_{ij}-1} \check{\kappa}_{j,\check{f}_{ij}m} \epsilon_{\check{d}_i}^{\check{f}_{ji}m} \bfe_i \right).
   \end{align}
   Since $i \not\in [1,l]$, the description of ${}^t \psi$ shows 
   \[
      {}^t \psi \left( \sum_{m=0}^{\check{d}_{ij}-1} \check{\kappa}_{j,\check{f}_{ij}m} \epsilon_{\check{d}_i}^{\check{f}_{ji}m} \bfe_i \right)
      = \sum_{m=0}^{\check{d}_{ij}-1} \check{\kappa}_{j,\check{f}_{ij}m} \epsilon_{\check{d}_i}^{\check{f}_{ji}m} \bfe_i.
   \]
   On the other hand, 
   \[
      \tilde{s}_i (\bmkap) = \bmkap - \sum_{j \in I} c_{ij} \sum_{m=0}^{d_{ij}-1} \kappa_{j,f_{ij}m} \epsilon_{d_i}^{f_{ji}m} \bfe_i.
   \]
   Therefore it is sufficient to show 
   \[
      \sum_{j \in I} \check{c}_{ij} \sum_{m=0}^{\check{d}_{ij}-1} \check{\kappa}_{j,\check{f}_{ij}m} \epsilon_{\check{d}_i}^{\check{f}_{ji}m} 
      = \sum_{j \in I} c_{ij} \sum_{m=0}^{d_{ij}-1} \kappa_{j,f_{ij}m} \epsilon_{d_i}^{f_{ji}m}.
   \]
   If $i \neq 0$, then 
   \[
      \sum_{j \in [0,l]} \check{c}_{ij} \sum_{m=0}^{\check{d}_{ij}-1} \check{\kappa}_{j,\check{f}_{ij}m} \epsilon_{\check{d}_i}^{\check{f}_{ji}m}
      = \check{c}_{i0} \sum_{j \in [0,l]} \check{\kappa}_j = c_{i0} \kappa_0,
   \]
   and hence 
   \begin{align}
      \sum_{j \in I} \check{c}_{ij} \sum_{m=0}^{\check{d}_{ij}-1} \check{\kappa}_{j,\check{f}_{ij}m} \epsilon_{\check{d}_i}^{\check{f}_{ji}m} 
      &= \sum_{j \not\in [0,l]} \check{c}_{ij} \sum_{m=0}^{\check{d}_{ij}-1} \check{\kappa}_{j,\check{f}_{ij}m} \epsilon_{\check{d}_i}^{\check{f}_{ji}m} + c_{i0} \kappa_0 \\ 
      &= \sum_{j \not\in [1,l]} c_{ij} \sum_{m=0}^{d_{ij}-1} \kappa_{j,f_{ij}m} \epsilon_{d_i}^{f_{ji}m} \\
      &= \sum_{j \in I} c_{ij} \sum_{m=0}^{d_{ij}-1} \kappa_{j,f_{ij}m} \epsilon_{d_i}^{f_{ji}m}.
   \end{align}
   If $i=0$, then  
   \begin{align}
      \sum_{j \in [0,l]} \check{c}_{ij} \sum_{m=0}^{\check{d}_{ij}-1} \check{\kappa}_{j,\check{f}_{ij}m} \epsilon_{\check{d}_i}^{\check{f}_{ji}m}
      &= 2\check{\kappa}_0 + (2-d) \sum_{j \in [1,l]} \check{\kappa}_j \\ 
      &= 2\kappa_0 -d \kappa_1 = \sum_{j \in [0,l]} c_{ij} \sum_{m=0}^{d_{ij}-1} \kappa_{j,f_{ij}m} \epsilon_{d_i}^{f_{ji}m},
   \end{align}
   while 
   \[
      \sum_{j \in I} \check{c}_{ij} \sum_{m=0}^{\check{d}_{ij}-1} \check{\kappa}_{j,\check{f}_{ij}m} \epsilon_{\check{d}_i}^{\check{f}_{ji}m} 
      = \sum_{j \in I} c_{ij} \sum_{m=0}^{d_{ij}-1} \kappa_{j,f_{ij}m} \epsilon_{d_i}^{f_{ji}m}.
   \]
   Thus we obtain the desired equality.
\end{proof}

\subsection*{Acknowledgements.} 
We are grateful to Tamas Hausel for kindly answering some questions on his work with M.~L.~Wong and D.~Wyss, 
and Yoshiyuki Kimura for valuable comments.   
\subsection*{Data Availability}
No datasets were generated or analysed during the current study.
\subsection*{Funding}
The second named author was supported by JSPS KAKENHI Grant Number 18K03256.
\subsection*{Declarations}
\textbf{Competing interest} The authors declare no competing interests. 
\bibliographystyle{plain}

\end{document}